\documentclass[a4paper,11pt,reqno]{amsart}
\usepackage{amsfonts}
\usepackage{amsmath,amsthm}
\usepackage{enumerate}
\usepackage[normalem]{ulem}
\usepackage{hyperref}
\usepackage[T1]{fontenc}
\usepackage[utf8]{inputenc}
\usepackage{mathtools}
\usepackage{mathrsfs}
\usepackage{amssymb} 
\usepackage[text={6.0in,8.6in},centering]{geometry}

\usepackage{color}
\usepackage{url}

\newtheorem{theorem}{Theorem}
\newtheorem{lemma}{Lemma}
\newtheorem{proposition}{Proposition}
\newtheorem{definition}{Definition}

\newtheorem{remark}{Remark}
\newtheorem{corollary}{Corollary}

\newtheorem*{examples*}{Examples}

\newcommand{\dd}{{\mathrm d}}

\def\R{\mathbb{R}}

\def\Mi{\mathscr{M}}

\def\di{\mathsf{d}}
\def\con{\mathfrak{C}}

\DeclareMathOperator{\hk}{H\HKkern K}
\newcommand{\HKkern}{%
  \mkern-6.0mu
  \mathchoice{}{}{\mkern0.2mu}{\mkern0.5mu}%
}

\DeclareMathOperator{\et}{E\ETkern T}
\newcommand{\ETkern}{%
  \mkern-3.0mu
  \mathchoice{}{}{\mkern0.2mu}{\mkern0.5mu}%
}

\def\ghk{\mathsf{G\kern-3pt\hk}}
\def\Det{\mathsf{D}_{\et}}
\def\DET{\mathbf{D}_{\et}}

\begin{document}
\title{Entropy-Transport distances between unbalanced metric measure spaces}
\author{Nicol\`o De Ponti}\thanks{Nicol\`o De Ponti: Scuola Internazionale Superiore di Studi Avanzati (SISSA), Trieste,  Italy, \\email: ndeponti@sissa.it} 
\author{Andrea Mondino}\thanks{Andrea Mondino: Mathematical Institute, University of Oxford, UK, \\email: Andrea.Mondino@maths.ox.ac.uk}

\begin{abstract}
Inspired by the recent theory of Entropy-Transport problems and by the $\mathbf{D}$-distance of Sturm on \emph{normalised} metric measure spaces, we define a new class of complete and separable distances between metric measure spaces of possibly different total mass. 
\\We provide several explicit examples of such distances, where a prominent role is played by a geodesic metric based on the Hellinger-Kantorovich distance. Moreover, we discuss some limiting cases of the theory, recovering the ``pure transport'' $\mathbf{D}$-distance and introducing a new class of ``pure entropic'' distances. 
\\We also study in detail the topology induced by such Entropy-Transport metrics, showing some compactness and stability results for metric measure spaces satisfying Ricci curvature lower bounds in a synthetic sense. 
\end{abstract}

\maketitle

\tableofcontents

\section*{Introduction}
With motivations from  pure Mathematics as well as from applied sciences, over the last decades a growing attention has been paid to the problem of ``comparing  objects'', which come naturally endowed with a distance/metric and a weight/volume form/measure. From the mathematical point of view, such objects are formalised as metric measure spaces (m.m.s. for short) $(X,\di,\mu)$, where the metric structure $(X,\di)$ describes the geometry and the mutual distance of points, and the measure $\mu$ ``weights'' the relative importance of different parts of the object. 
\\The flexibility of such a framework allows to unify the treatment of a series of problems stemming from various fields of science and technology, e.g. chemistry \cite{KM}, data science \cite{PC},  multi-omics data alignment \cite{DSSSS}, computer vision \cite{Sch}, language processing \cite{A-MJ}, graph \cite{XLC} and shape \cite{SPKS, WPTM} matching, barycenters \& shape analysis \cite{PCS}, Generative Adversarial Networks \cite{BAKJ},  machine learning \cite{VCFTC}. The theory of metric measure spaces has been flourishing in pure Mathematics as well, providing a unified setting to investigate concentration of measure phenomena \cite{Ledoux, Shioya}, the theory of Ricci limit spaces \cite{Fukaya, CC97I} and, more generally, synthetic notions of Ricci curvature lower bounds \cite{St1, St2, LV, AGS}.

In order to ``quantify the similarities and differences between two such objects'', it is thus  natural to investigate appropriate notions of distance between metric measure spaces. This idea has its roots in the work of Gromov \cite[Chapter 3$\frac{1}{2}$]{Grom}), who first recognized the importance of studying the ``space of spaces'' $\boldsymbol{\mathrm{X}}$ as a metric space in its own right. Formally, $\boldsymbol{\mathrm{X}}$ denotes the set of equivalence classes of metric measure spaces $(X,\di,\mu)$, where $(X,\di)$ is a complete and separable metric space, and $\mu$ is a finite, nonnegative, Borel measure;  we are naturally identifying two m.m.s. $(X_1,\mathsf{d}_1,\mu_1)$, $(X_2,\mathsf{d}_2,\mu_2)$ if there exists an isometry $\psi:\mathsf{supp}(\mu_1)\rightarrow \mathsf{supp}(\mu_2)$ such that $\psi_{\sharp}\,\mu_1=\mu_2$. Here by $\mathsf{supp}(\mu)$ we denote the support of the measure $\mu$ (see the preliminary section for more details).  
\\In the recent years, the theory has been pushed forward by the works of Sturm \cite{St1, St3} and Memoli \cite{Memoli1} who realized that ideas from mass transportation can be used to produce new relevant distances between metric measure spaces. Such distances have been successfully applied in different fields, but suffer from a major restriction which is intrinsic of the Wasserstein distances coming from optimal transport: they can be used to compare only spaces \emph{with the same total mass}.

The goal of the present paper is to overcome this limitation by taking advantage of the theory of optimal \emph{Entropy-Transport} problems \cite{LMS}. In contrast with the classical transport setting, these problems allow the description of phenomena where the conservation of mass may not hold; for this reason they are also known in the literature as ``unbalanced optimal transport problems''.  The corresponding theory is fairly recent and is becoming increasingly popular in applications, e.g. gradient flows to train neural networks \cite{CB,RJBV}, supervised learning \cite{FZMA}, medical imaging \cite{FRTG} and video \cite{LBR} registration.  Indeed, the Entropy-Transport relaxation  seems to outperform classical optimal transport in all the problems where the input data is noisy or a normalization procedure is not appropriate.
 We refer to \cite{SFVTP} and references therein for more applications of unbalanced optimal transport. 

As we are going to explain in detail below, inspired by the construction of the $\mathbf{D}$-distance of Sturm \cite{St1},  we are able to produce a new class of complete and separable distances between metric measure spaces by replacing the Wasserstein distance with an Entropy-Transport distance. Such metric structures on $\boldsymbol{\mathrm{X}}$ also turn out to be geodesic (resp. length) when the underlying Entropy-Transport distance is geodesic (resp. length).

\vspace{5mm}
\textbf{Optimal transport and Sturm distances.} Let $(X,\di)$ be a metric space and $\boldsymbol{\mathrm c}:X\times X\rightarrow [0,+\infty]$ be a lower semi-continuous cost function.  The optimal transport problem between two probability measures $\mu_1,\mu_2$ consists in the minimization problem:
\begin{equation}\label{intro: Kantorovich ot problem}
\mathrm{T}(\mu_1,\mu_2):=\inf_{\boldsymbol{\gamma}\in \Pi(\mu_1,\mu_2)}\int_{X\times X}\boldsymbol{\mathrm c}(x_1,x_2)\, \dd \boldsymbol{\gamma}(x_1,x_2).
\end{equation}
Here $\Pi(\mu_1,\mu_2)$ denotes the set of measures $\boldsymbol{\gamma}$ in the product space $X\times X$ whose marginals satisfy the constraint $\pi^i_{\sharp}\boldsymbol{\gamma}=\mu_i$, where $\pi^i$ denotes the projection map $\pi^i(x_1,x_2)=x_i$.
\\A typical choice for the cost function is $\boldsymbol{\mathrm c}(x_1,x_2)=\di^p(x_1,x_2)$, $p\ge 1$. In this situation, the transport cost $\mathrm{T}$ is the $p$-power of the celebrated $p$-Wasserstein distance $\mathcal{W}_p$, a metric on the set $\mathscr{P}_p(X)$ of probability measures over $X$ with finite $p$-moment. Starting from the seminal work of Kantorovich, the metric space $(\mathscr{P}_p(X), \mathcal{W}_p)$ has been thoroughly studied: it inherits many geometric properties of the underlying space $(X,\di)$ (such as completeness, separability, geodesic property) and induces the weak topology (with $p$-moments) of probability measures.  We refer to the monograph \cite{Vil} for a detailed overview of the topic.

As observed by Sturm \cite{St1}, one can lift the metric $\mathcal{W}_p$ to a distance between metric measure spaces by defining:
\begin{equation}
\mathbf{D}_p\big((X_1,\mathsf{d}_1,\mu_1),(X_2,\mathsf{d}_2,\mu_2)\big):=\inf \mathcal{W}_p(\psi^1_{\sharp}\mu_1,\psi^2_{\sharp}\mu_2),
\end{equation}
where the infimum is taken over all complete and separable metric spaces $(\hat{X},\hat{\mathsf{d}})$, and all isometric embeddings $\psi^i:\mathsf{supp}(\mu_i)\rightarrow \hat{X}$. It is proved in \cite[Theorem 3.6]{St1} that $\mathbf{D}_p$ is a complete, separable and geodesic distance on the set
$$\boldsymbol{\mathrm{X}}_{1,p}:=\{ (X,\mathsf{d},\mu)\in \boldsymbol{\mathrm{X}}: \ \mu\in \mathscr{P}_p(X,\mathsf{d})\}.$$

\vspace{5mm}
\textbf{Entropy-Transport problems and Sturm-Entropy-Transport distances.} 
The idea at the core of Entropy-Transport problems is to relax the marginal constraints typical of the classical Kantorovich formulation \eqref{intro: Kantorovich ot problem} by adding some suitable penalizing functionals which keep track of the deviation of the marginals $\gamma_i:=\pi^i_{\sharp}\boldsymbol{\gamma}$ from the data $\mu_i$, $i=1,2$. 
\\Following the approach of Liero, Mielke and Savaré \cite{LMS}, given a superlinear, convex function $F:[0,+\infty)\rightarrow [0,+\infty]$ such that $F(1)=0$ (for simplicity here we assume $F$ to be superlinear, see definition \eqref{def: F divergenza} for the general case), one considers the entropy functional (also called Csisz\'ar $F$-divergence \cite{Csiszar})
\begin{equation}
D_F:\mathscr{M}(X)\times \mathscr{M}(X)\rightarrow [0,+\infty], \qquad
D_F(\gamma||\mu):=\begin{cases} \int_X F\big(\frac{\mathrm{d}\gamma}{\mathrm{d}\mu}\big)\mathrm{d} \mu \quad &\textrm{if} \ \gamma \ll \mu, \\
+\infty &\textrm{otherwise}. \end{cases}
\end{equation}
Here $\mathscr{M}(X)$ denotes the set of finite, nonnegative, Borel measures over $X$. 
A classical example is given by the choice $F=U_1(s):=s\ln (s)-s+1$, that corresponds to the celebrated Kullback-Leibler divergence (note that when ${\boldsymbol{\gamma}}$ and $\mu$ are probability measures, $D_{U_1}$ coincides with the celebrated Boltzmann-Shannon entropy ${\rm Ent}(\rho \mu| \mu)=\int \rho \log \rho \, \dd\mu$). 

Given $\mu_1,\mu_2\in \mathscr{M}(X)$, the Entropy-Transport problem induced by the entropy function $F$ and the cost function $\boldsymbol{\mathrm c}$ is then defined as  
\begin{equation}\label{intro: ET problem}
\et(\mu_1,\mu_2):=\inf_{\boldsymbol{\gamma}\in \mathscr{M}(X\times X)}\bigg\{\sum_{i=1}^2 D_F(\gamma_i||\mu_i)+\int_{X\times X}\boldsymbol{\mathrm c}(x_1,x_2)\dd \boldsymbol{\gamma}(x_1,x_2)\bigg\}.
\end{equation}
We emphasize that the problem \eqref{intro: ET problem} makes perfect sense even when $\mu_1(X)\neq \mu_2(X)$. 

As in the case of optimal transport problems, it is natural to consider cost functions of the form $\boldsymbol{\mathrm c}(x_1,x_2)=\ell(\di(x_1,x_2))$, where $\di$ is a distance on $X$ and $\ell:=[0,\infty)\rightarrow [0,\infty]$ is a general function. With a careful choice of the functions $F$ and $\ell$ (see \cite{DepIeee} for a discussion on the metric properties of Entropy-Transport problems), one is able to produce a distance $\Det$ on the space $\mathscr{M}(X)$ by taking a suitable power of the Entropy-Transport cost $\et$, namely $\Det=\et^a$ for a certain $a\in (0,1]$. 
\\In the paper we introduce the class of \emph{regular} Entropy-Transport distances. The formal definition of this class of distances is given in Definition \ref{def: Entropy-Transport distance}, here we just mention than any regular Entropy-Transport distance $\Det$ is a complete and separable metric on $\mathscr{M}(X)$ of the form $\Det=\et^a$, for an Entropy-Transport cost induced by sufficiently regular functions $F$ and $\ell$.

For any regular Entropy-Transport distance, the \emph{Sturm-Entropy-Transport} distance $\DET$ between the (equivalence classes of) m.m.s. $(X_1,\di_1,\mu_1)$, $(X_2,\di_2,\mu_2)$ is then defined as
\begin{equation}\label{intro: def DET}
\DET\big((X_1,\mathsf{d}_1,\mu_1),(X_2,\mathsf{d}_2,\mu_2)\big):=\inf \Det (\psi^1_{\sharp}\mu_1,\psi^2_{\sharp}\mu_2)
\end{equation}
where the infimum is taken over all complete and separable metric spaces $(\hat{X},\hat{\mathsf{d}})$, and all isometric embeddings $\psi^1:\mathsf{supp}(\mu_1)\rightarrow \hat{X}$ and $\psi^2:\mathsf{supp}(\mu_2)\rightarrow \hat{X}$.

The main result of the paper (Theorem \ref{th main}) is that every Sturm-Entropy-Transport distance defines a complete and separable metric structure on $\mathbf{X}$. Moreover it satisfies the geodesic (resp. length) property if the distance $\Det$ satisfies the geodesic (resp. length) property on the space of measures.

 We also study in detail the notion of convergence induced by such distances, showing that it corresponds to the weak measured-Gromov convergence introduced in \cite{GMS}. As a consequence, we obtain a compactness result for the class of m.m.s. $(X,\mathsf{d},\mu)$ satisfying the $\mathsf{CD}(K,N)$ condition, having bounded diameter and satisfying  $0<v\le \mu(X)\le V$. We refer to Theorem \ref{th: compact CD} for the precise statement and to the preliminaries for the definition of the curvature-dimension condition $\mathsf{CD}(K,N)$.

At a technical level, the proofs of our results are inspired by the corresponding ones given by Sturm in \cite{St1}, but they require new ideas in order to deal with general cost functions and with the entropic part of the problem. Two key results of independent interest are contained in Proposition \ref{pr: formulazione D-LET} and Lemma \ref{lem:esistenza ottimo}, where we show that the infimum in the right hand side of \eqref{intro: def DET} is actually a minimum, and we give an explicit formulation of the Sturm-Entropy-Transport distance, namely
\begin{equation}
\DET^{1/a}((X_1,\mathsf{d}_1,\mu_1),(X_2,\mathsf{d}_2,\mu_2))=\sum_{i=1}^2D_F(\gamma_i||\mu_i)+\int_{X_1\times X_2}\ell\big(\hat{\mathsf{d}}(x,y)\big)\dd\boldsymbol{\gamma},
\end{equation}
for some optimal measure $\boldsymbol{\gamma}\in\mathscr{M}(X_1\times X_2)$ and optimal pseudo-metric coupling $\hat{\mathsf{d}}$ between $\mathsf{d}_1$ and $\mathsf{d}_2$ (see the preliminaries for the definition of pseudo-metric coupling).
Also the proof of one of the main results, Theorem \ref{th main}, despite being inspired by \cite{St1}, departs from it and needs some new ideas:
\begin{itemize}
\item in order to show that $\mathbf{D}_p$ defines a \emph{non-degenerate} distance function (i.e. $$\mathbf{D}_p \big( (X_1,\mathsf{d}_1,\mu_1),  (X_2,\mathsf{d}_2,\mu_2)  \big)=0$$ implies that $(X_1,\mathsf{d}_1,\mu_1)$ and $(X_2,\mathsf{d}_2,\mu_2)$ are isomorphic as metric measure spaces), Sturm \cite{St1} establishes a comparison result with Gromov's $\underline\Box_{1}$ distance, of independent interest; this permits to inherit the non-degeneracy of $\mathbf{D}_p$ by the one of $\underline\Box_{1}$.
\\Instead, we argue directly: thanks to the aforementioned Proposition \ref{pr: formulazione D-LET} and Lemma \ref{lem:esistenza ottimo}, we can exploit the existence of an optimal coupling \emph{both at the level of space and measure} and  infer the non-degeneracy of $\DET$ directly;
\item  in order to show that the $\mathbf{D}_p$ distance is \emph{length}, Sturm \cite{St1} argues by approximation via finite metric spaces, taking advantage of the ``pure transport'' behaviour of $\mathbf{D}_p$.
\\ Due to the entropy contribution in the  $\DET$ distance, we argue differently:  the main point is to embed everything in a complete, separable and \emph{geodesic} ambient space, obtained by a slight modification of Kuratowski embedding.
 \end{itemize}

The class of regular Entropy-Transport distances includes some of the main examples of Entropy-Transport distances known in the literature, including:
\begin{itemize}
\item The Hellinger-Kantorovich geodesic distance \cite{LMS, LMS1, Chizat, KMV} induced by the choices
$$a=1/2\, , \qquad F(s)=U_1(s)\, , \qquad \ell(d)=\begin{cases} -\log\left({\cos^2(d)}\right) \ \ &\textrm{if} \ d<\frac{\pi}{2}, \\
+\infty \ \ &\textrm{otherwise}.\end{cases}$$
\vspace{5mm}
\item The so-called \emph{Gaussian} Hellinger-Kantorovich distance \cite{LMS} that corresponds to the choices
$$a=1/2\, , \qquad F(s)=U_1(s)\, , \qquad \ell(d)=d^2.$$
\vspace{5mm}
\item The \emph{quadratic power-like} distances studied in \cite{DepIeee} corresponding to
$$a=1/2\, , \qquad F(s)=U_p(s):=\frac{s^p-p(s-1)-1}{p(p-1)}\, , \qquad \ell(d)=d^2, \qquad 1<p\le 3.$$
\end{itemize}

\vspace{3mm}

Moreover, our analysis is not restricted to regular Entropy-Transport distances. By a limit procedure we also discuss some singular cases covering:
\begin{itemize}
\item The ``pure entropy'' setting that corresponds to the choice
$$\boldsymbol{\mathrm c}(x_1,x_2)=\begin{cases}0 \ &\textrm{if} \ x_1=x_2, \\ +\infty &\textrm{otherwise.} \end{cases}$$
In this situation we construct a family of distances between metric measure spaces inducing a notion of \emph{strong} convergence (see Theorems \ref{th: limit entropy distances} and \ref{th: SPL vs mG convergence} for the details). 
\vspace{5mm}
\item The ``pure transport'' setting, corresponding to
$$a=1/p\, , \qquad F(s)=\begin{cases} 0 \ &\mathrm{if} \ s=1, \\ +\infty &\mathrm{otherwise}, \end{cases} \qquad \ell(d)=d^p,$$
where we recover the $\mathbf{D}_p$-distances introduced by Sturm.
\vspace{5mm}
\item The \emph{Piccoli-Rossi} distance $\mathsf{BL}$ \cite{PR1,PR2} (also known as \emph{bounded-Lipschitz distance}), induced by the choices
$$a=1\, , \qquad F(s)=|s-1|\, , \qquad \ell(d)=d.$$
By an analogous procedure to the one described in \eqref{intro: def DET}, in Theorem \ref{th: limit PR} we show that the distance $\mathsf{BL}$ can be lifted to a complete distance $\mathbf{BL}$ on the set $\mathbf{X}$. 
\end{itemize} 
 
\vspace{5mm}
\textbf{Note on the preparation.} Some of the results of the paper (often under additional assumptions) have been presented at different seminars and included in the Phd thesis of the first named author \cite[Chapter 5]{DepPhd}, where the construction of the Sturm-Entropy-Transport distances induced by the Hellinger-Kantorovich and the quadratic power-like distances is developed. 
\\Only during the final stage of preparation of the present manuscript (September 2020), we became aware of the independent work \cite{SVP}, which defines a class of distances (denoted by $\mathrm{CGW}$, for ``conic Gromow-Wasserstein'') between unbalanced metric measure spaces starting from the construction of the Gromov-Wasserstein distance introduced in \cite{Memoli1} and the \emph{conical} formulation of the Entropy-Transport problems (see \cite{LMS, Chizat, DepIeee} and Remark \ref{rem: conical formulation} for a discussion on the “cone geometry'' of Entropy-Transport problems). The paper \cite{SVP} also provides some interesting numerical discussions on the topic, while it is not present a study on the analytic and geometric properties of this class of distances (such as completeness, separability, the length and geodesic property, compactness).
An expert reader will notice that our $\DET$ distance is an \emph{unbalanced} counterpart of Sturm's $\mathbf{D}_p$ distance \cite{St1} for probability metric measure spaces, while S\'ejourn\'e-Vialard-Peyr\'e $\mathrm{CGW}$ distance \cite{SVP} is an \emph{unbalanced} counterpart of Memoli's \cite{Memoli1} Gromov-Wasserstein distance. A major difference between the two approaches is that while our $\DET$ distance is complete (see Theorem \ref{th main}),  the Gromov-Wasserstein distance of \cite{Memoli1}  is not complete, and the same is expected for the $\mathrm{CGW}$ distance of \cite{SVP}. The relation between $\DET$ and $\mathrm{CGW}$ is analysed in Section \ref{sec: CGW}, where we prove an upper bound of the latter in terms of the former.

\vspace{5mm}
\textbf{Acknowledgements.}
The project started when N.D.P. was visiting A.M. in the fall 2018 at the Mathematics Institute of the University of Warwick, and took advantage of a second visit of N.D.P.  to the Mathematical Institute of the University of Oxford in March 2020. The authors wish to thank both the institutions for the inspiring atmosphere and the excellent working conditions.
\\A.M. acknowledges the support of the EPSRC First Grant  EP/R004730/1 “Optimal transport and geometric
analysis” and of the European Research Council (ERC), under the European's Union Horizon 2020 research and innovation programme, via the ERC Starting Grant  “CURVATURE”, grant agreement No. 802689.
\\The authors wish to warmly thank Giuseppe Savaré for valuable discussions on the topics of the paper and are grateful to the anonymous reviewers, for their suggestions and comments that helped to improve
a previous version of the manuscript.

\section{Preliminaries and notation}
\subsection{Metric and measure setting}
A function $\mathsf{d}:X\times X\rightarrow [0,\infty]$ is a \emph{pseudo-metric} on the set $X$ if $\mathsf{d}$ is symmetric, satisfies the triangle inequality and $\mathsf{d}(x,x)=0$ for every $x\in X$. We say that $\mathsf{d}$ is a metric possibly attaining the value $+\infty$ if it is a pseudo-metric such that $\mathsf{d}(x,y)=0$ implies $x=y$. When $\mathsf{d}$ is also finite-valued, we simply say that $\mathsf{d}$ is a metric.
A \emph{pseudo-metric space} (resp. \emph{metric space}) will be a couple $(X,\mathsf{d})$, where $\mathsf{d}$ is a pseudo-metric (resp. metric) on the set $X$.

On a pseudo-metric space we will always consider the topology induced by the open balls $B_r(x):=\{y\in X: \mathsf{d}(x,y)<r\}.$ A \emph{Polish space} is a separable completely metrizable topological space. 
We will denote by $\mathsf{diam}(X)$ the diameter of a metric space $X$.

An \emph{isometry} between two metric spaces $(X_1,\mathsf{d_1})$, $(X_2,\mathsf{d}_2)$ is a map $\psi:X_1\rightarrow X_2$ such that for every $x,y\in X_1$ we have
\begin{equation}
\mathsf{d}_1(x,y)=\mathsf{d}_2(\psi(x),\psi(y)).
\end{equation}

Let $\{(X_{\alpha},\mathsf{d}_{\alpha})|{\alpha}\in A\}$ be an indexed family of metric spaces, we define its \emph{disjoint union} as 
\begin{equation*}
\bigsqcup_{\alpha}X_{\alpha}:=\bigcup\Big\{X_{\alpha}\times \{\alpha\}|{\alpha}\in A\Big\},
\end{equation*}
endowed with a pseudo-metric $\hat{\mathsf{d}}$, called \emph{pseudo-metric coupling} between $\{\mathsf{d}_{\alpha}\}$, such that $\hat{\mathsf{d}}((x,\alpha),(y,\alpha))=\mathsf{d}_{\alpha}(x,y)$ for every $x,y\in X_{\alpha}$.
The inclusion map 
\begin{equation*}
\iota_{\alpha}:X_{\alpha}\to \bigsqcup_{\alpha}X_{\alpha}, \ \ \ \iota_{\alpha}(x):=(x,\alpha),
\end{equation*}
is thus an isometry with image $X_{\alpha}\times\{\alpha\}$. We will often identify, with a slight abuse of notation, the space $X_{\alpha}$ with $X_{\alpha}\times\{\alpha\}$. 

\begin{lemma}\label{lem: quotient disjoint union is metric}
Let $(X_1,\di_1)$, $(X_2,\di_2)$ be two complete and separable metric spaces. Let $\hat{\di}$ be a finite valued pseudo-metric coupling between $\di_1$ and $\di_2$. Then the space 
\begin{equation}
\tilde{X}:=(X_1\sqcup X_2)/\sim \qquad \textit{where} \qquad x_1\sim x_2 \Longleftrightarrow \hat{\di}(x_1,x_2)=0
\end{equation} 
endowed with the distance 
$$\tilde{\di}([x_1],[x_2]):=\hat{\di}(x_1,x_2)$$
is a complete and separable metric space. Here $[x]\in \tilde{X}$ denotes the equivalence class of the point $x\in X_1\sqcup X_2$.
\end{lemma}
\begin{proof}
We firstly notice that $\tilde{\di}$ is well defined on $\tilde{X}$. Indeed, if $x_1 \sim \tilde{x}_1$ and $x_2 \sim \tilde{x}_2$ we have
\begin{equation*}
\begin{aligned}
&\tilde{\di}(x_1,x_2)=\hat{\di}(x_1,x_2)\leq \hat{\di}(x_1,\tilde{x}_1)+\hat{\di}(\tilde{x}_1,\tilde{x}_2)+\hat{\di}(\tilde{x}_2, x_2) =\tilde{\di}(\tilde{x}_1,\tilde{x}_2)\\
&\tilde{\di}(\tilde{x}_1,\tilde{x}_2)=\hat{\di}(\tilde{x}_1,\tilde{x}_2)\leq \hat{\di}(\tilde{x}_1, x_1)+\hat{\di}(x_1,x_2)+\hat{\di}(x_2,\tilde{x}_2)=\tilde{\di}(x_1,x_2)
\end{aligned}
\end{equation*}
which implies $\tilde{\di}(x_1,x_2)=\tilde{\di}(\tilde{x}_1,\tilde{x}_2)$.
\\It is clear that $\tilde{\di}$ is a metric on $(X_1\sqcup X_2)/\sim$. 
\\The separability is a consequence of the fact that $(X_1\sqcup X_2,\hat{\di})$ is separable, being the union of two separable space (recall that $\hat{\di}=\di_i$ on $X_i$, $i=1,2$).
\\ To prove the completeness, let us consider a Cauchy sequence $\{y_j\}\in \tilde{X}$. It is sufficient to show that a subsequence is converging with respect to $\tilde{\di}$. Let $p:X_1\sqcup X_2\rightarrow \tilde{X}$ be the quotient map and, recalling that $X_1\sqcup X_2=X_1\times \{0\}\cup X_2\times\{1\}$, we can suppose without loss of generality that there exists a subsequence $\{p^{-1}(y_{j_k})\}\in X_1\times\{0\}$ (the case $\{p^{-1}(y_{j_k})\}\in X_2\times\{1\}$ being analogous). Up to identifying $(X_1\times\{0\},\hat{\di})$ with $(X_1,\di_1)$, we can infer that $\{p^{-1}(y_{j_k})\}$ is a Cauchy sequence in the complete space $(X_1,\di_1)$ and thus it converges. It is immediate to check that $\{y_{j_k}\}$ is converging in $\tilde{X}$ with respect to $\tilde{\di}$ and the proof is complete. 
\end{proof}

Starting from a metric space $(X,\mathsf{d})$, we define the \emph{cone over $X$} as the space 
$$\mathfrak{C}(X):=(X\times [0,+\infty))/{\sim} \quad \textrm{where} \quad  (x_1,r_1)\sim (x_2,r_2) \iff r_1=r_2=0 \ \mbox{or} \ r_1=r_2, x_1=x_2.$$

If $(X,\mathsf{d})$ is a pseudo-metric space, we denote by $\mathscr{M}(X)$ the space of finite, nonnegative measures on the Borel $\sigma$-algebra $\mathscr{B}(X)$, and by $\mathscr{P}(X)\subset \mathscr{M}(X)$ the space of probability measures. 
We endow $\mathscr{M}(X)$ with the weak topology, inducing the following notion of convergence:
\begin{equation}
\mu_n\rightharpoonup \mu \iff \int_X f\dd\mu_n \rightarrow \int_X f\dd\mu \ \ \textrm{for any} \ f\in C_b(X),
\end{equation} 
where $C_b(X)$ denotes the set of real, continuous and bounded functions defined on $X$.

A subset $\mathscr{K}\subset \mathscr{M}(X)$ is \emph{bounded} if $\sup_{\mu\in \mathscr{K}}\mu(X)<\infty$ and it is \emph{equally tight} if
\begin{equation}
\forall \epsilon >0 \ \ \exists K_{\epsilon}\subset X \ \textrm{compact}: \ \forall \mu\in \mathscr{K}, \ \ \mu(X\setminus K_{\epsilon})\leq \epsilon.
\end{equation}

Compactness properties with respect to the weak topology on $\mathscr{M}(X)$ are guaranteed by the following version of Prokhorov's Theorem:
\begin{theorem}\label{th: Prokhorov}
Let $X$ be a Polish space. A subset $\mathscr{K}\subset \mathscr{M}(X)$ is bounded and equally tight if and only if it is relatively compact with respect to the weak topology.
\end{theorem} 

We recall that the set of measures of the form
\begin{equation}\label{def: somme di delta}
\mu=M\sum_{n=1}^N \delta_{x_n},
\end{equation}
where $M\in \mathbb{R}_+$, $N\in \mathbb{N}$ and $x_n\in X$, is dense in $\mathscr{M}(X)$. Moreover, if $X$ is separable, the measures of the form \eqref{def: somme di delta}, with $M\in \mathbb{Q}_+$ and $x_n$ in a countable dense subset of $X$, form a countable dense subset of $\mathscr{M}(X)$, proving that also the latter is a separable space.

A \emph{metric measure space} will be a triple $(X,\mathsf{d},\mu)$ where $(X,\mathsf{d})$ is a complete, separable metric space and $\mu\in \mathscr{M}(X)$.  
If there exists a point $x_0\in X$ such that 
\begin{equation}
\int_X \mathsf{d}^p(x_0,x)\dd\mu(x)<\infty,
\end{equation}
we will say that the measure $\mu\in \mathscr{P}(X)$ has finite $p$-moment. We denote by $\mathscr{P}_p(X)$ the space of measures $\nu\in\mathscr{P}(X)$ with finite $p$-moment.

The support of the measure $\mu$ is the smallest closed set $X_0:=\mathsf{supp}(\mu)$ such that $\mu(X\setminus X_0)=0$. We notice that the set $\mathsf{supp}(\mu)$ has a natural structure of metric measure space with the induced distance, $\sigma$-algebra and measure (which will be denoted in the same way). 

We say that $\varphi$ is a \emph{curve} connecting $x,y\in X$, if $\varphi:[a,b]\rightarrow X$ is a continuous map such that $\varphi(a)=x$ and $\varphi(b)=y$. The length of a curve is defined as
\begin{equation}\label{eq:defLength}
\mathsf{Length}(\varphi):=\sup \sum_{i=1}^n \mathsf{d}\big(\varphi(t_{i-1}),\varphi(t_i)\big),
\end{equation}
where the supremum is taken over all the partitions $a=t_0<t_1<...<t_n=b$.

We will always assume that a curve of finite length is parametrized by constant speed, i.e.
\begin{equation}
\mathsf{Length}(\varphi\restriction_{[s,t]})=\frac{t-s}{b-a}\mathsf{Length}(\varphi).
\end{equation}
A metric space $(X,\mathsf{d})$ is called \emph{length space} if for all $x,y\in X$
\begin{equation}
\mathsf{d}(x,y)=\inf\Big\{\mathsf{Length}(\varphi):\varphi \ \textrm{curve connecting} \ x \ \textrm{and} \ y\Big\}. 
\end{equation}
A \emph{geodesic} is a curve $\varphi:[a,b]\rightarrow X$ such that 
$$\di(\varphi(a), \varphi(t))= (t-a)\,  \di(\varphi(a), \varphi(b)), \quad \text{for all } t\in [a,b] .$$
 Notice in particular that if $\varphi$ is a geodesic then 
 $$\mathsf{Length}(\varphi)=\mathsf{d}(\varphi(a),\varphi(b)).$$ 
A metric space $(X,\mathsf{d})$ is geodesic if any pair of points $x,y\in X$ is connected by a geodesic. 

For a metric space $(X,\mathsf{d})$, the Kantorovich-Wasserstein distance $\mathcal{W}_p$ of order $p$, $p\ge 1$, is defined as follows:  for $\mu_0,\mu_1 \in \mathscr{M}(X)$ we set
\begin{equation}\label{eq:Wdef}
\mathcal{W}^p_p(\mu_0,\mu_1) := \inf_{\boldsymbol{\gamma}} \int_{X\times X} \mathsf{d}^p(x,y) \, \dd\boldsymbol{\gamma},
\end{equation}
where the infimum is taken over all $\boldsymbol{\gamma} \in \mathscr{M}(X \times X)$ with $\mu_0$ and $\mu_1$ as the first and the second marginal, i.e. $(\pi^i)_{\sharp}\boldsymbol{\gamma}=\mu_i$ where $\pi^i:X\times X\rightarrow X$ denotes the projection map $\pi^i(x_1,x_2)=x_i$, $i=1,2$. A measure $\boldsymbol{\gamma} \in \mathscr{M}(X \times X)$  achieving the minimum in \eqref{eq:Wdef} with given marginals is said a $\mathcal{W}_p$-optimal coupling for $(\mu_0,\mu_1)$.
 It is clear that $\mathcal{W}_p(\mu_1,\mu_2)=+\infty$ when $\mu_1(X)\neq \mu_2(X)$. 

If $(X,\mathsf{d})$ is complete and separable, $(\mathscr{P}_p(X),\mathcal{W}_p)$ is a complete and separable metric space. It is geodesic when $(X,\mathsf{d})$ is geodesic. Moreover, for any sequence $\mu_n\in \mathscr{P}_p(X)$ we have
\begin{equation}
\lim_{n\to \infty}\mathcal{W}_p(\mu_n,\mu)=0 \iff 
\begin{cases}
&\mu_n \ \textrm{weakly converges to} \ \mu, \\
&\mu_n \ \textrm{has uniformly} \ p\textrm{-integrable moments},
\end{cases}
\end{equation} 
where the latter means that for some (thus any) $x_0$
\begin{equation}
\lim_{R\to \infty}\limsup_{n}\int_{X\setminus B_R(x_0)}\mathsf{d}^p(x_0,x)\dd\mu_n(x)=0.
\end{equation}
For a proof of these last facts, see \cite[Theorem 6.18]{Vil}.

\subsection{Curvature-Dimension condition}

It is out of the scopes of this brief section to give a full account of the curvature-dimension condition and its properties; we will limit to schematically recalling the basic definitions involved. The interested reader is referred to  the original papers \cite{LV, St1, St2, AGS, AGMR, G15, GMS, EKS, AMS, CaMi}, the survey \cite{AmbrosioICM} and the monograph \cite{Vil}. 
\begin{itemize}
	\item For any $K\in \R, \, N\in (1,\infty), \,\theta >0$ and $t\in[0,1]$, define the distortion coefficients by
		\begin{equation*}
		\tau^{(t)}_{K,N}(\theta) := t^{\frac{1}{N}} \sigma^{(t)}_{K, N-1}(\theta)^{\frac{N-1}{N}},
		\end{equation*}
		where 
		\begin{equation*}
		\sigma^{(t)}_{K,N}(\theta):= 
			\begin{cases}
				\infty
					&\text{if $K\theta^2\geq N\pi^2$}\\
				\frac{\sin(t\theta \sqrt{K/N})}{\sin(\theta \sqrt{K/N})}
					&\text{if $0<K\theta^2< N\pi^2$}\\
				t
					&\text{if $K\theta^2=0$}\\
				\frac{\sinh(t\theta \sqrt{K/N})}{\sinh(\theta \sqrt{K/N})}
					&\text{if $K\theta^2<0$}
			\end{cases}.
		\end{equation*}
	\item For every $N\in (1,\infty)$, define  the \emph{$N$-R\'enyi entropy} functional relative to $\mu$,  ${\mathcal U}_N(\cdot\,| \mu):{\mathscr P}(X) \to [-\infty,0]$  as
		\begin{equation*}
		{\mathcal U}_N(\nu| \mu):=- \int_{X} \rho^{1-\frac{1}{N}} \dd\mu, \quad \text{where $\nu = \rho\mu+\nu^s$ and $\nu^s\perp \mu$}.
		\end{equation*}
	\item Define also the \emph{Boltzmann-Shannon entropy} functional relative to $\mu$, ${\rm Ent}(\cdot\,| \mu):{\mathscr P}(X) \to (-\infty,+ \infty]$ as
	  	\begin{equation*}
		{\rm Ent}(\nu| \mu):= \int_{X} \rho \,\log(\rho)\,  \dd\mu, \quad \text{if $\nu = \rho\mu \ll \mu$ and $\rho \log \rho \in L^1(X,\mu)$},
		\end{equation*}
		and $+\infty$ otherwise.
	\item \emph{$\mathsf{CD}(K,\infty)$ condition}: given $K\in \R$, we say that $(X,\di,\mu)$ verifies the $\mathsf{CD}(K,\infty)$ condition  if   for any pair of probability measures  
	$\nu_0,\nu_1\in {\mathscr P}_2 (X) $ with 
	$${\rm Ent}(\nu_0| \mu), {\rm Ent}(\nu_1| \mu)<+\infty,$$ 
	there exists a $\mathcal{W}_2$-geodesic $(\nu_t)_{t\in[0,1]}$ from $\nu_0$ to $\nu_1$ such that 
		\begin{equation*}
		{\rm Ent}(\nu_t| \mu)\leq (1-t)\, {\rm Ent}(\nu_0| \mu)+ t\, {\rm Ent}(\nu_1| \mu)- \frac{K}{2} t (1-t) \mathcal{W}^2_2(\mu_0, \mu_1),
		\end{equation*}
		for any $t\in [0,1]$.
	\item \emph{$\mathsf{CD}(K,N)$ condition}: given $K\in \R$, $N\in(1,\infty)$ we say that $(X,\di,\mu)$ verifies the $\mathsf{CD}(K,N)$ condition  if for any pair of probability measures  
	$\nu_0,\nu_1\in {\mathscr P}{_2}(X) $ with bounded support and with $\nu_0,\nu_1\ll \mu$, there exists a $\mathcal{W}_2$-geodesic $(\nu_t)_{t\in[0,1]}$ from $\nu_0$ to $\nu_1$ with  $\nu_t \ll \mu$, and a $\mathcal{W}_2$-optimal coupling  $\boldsymbol{\gamma} \in {\mathscr P}(X \times X)$
	  such that
		\begin{equation*}
		{\mathcal U}_{N'}(\nu_t| \mu)\leq  - \int   \left[\tau^{(1-t)}_{K,N'}(\di(x,y))\rho_0^{-\frac{1}{N'}}+
				\tau^{(t)}_{K,N'}(\di(x,y))  \rho_1^{-\frac{1}{N'}} \right]  \dd\boldsymbol{\gamma} (x,y),
		\end{equation*}
		for any $N'\geq N$, $t\in[0,1]$.
\item \emph{Consistency property: } A smooth Riemannian manifold (resp. weighted Riemannian manifold) $M$ satisfies the $\mathsf{CD}(K,N)$ condition for some $K \in \R, N\in (1,\infty)$  if and only if ${\rm dim}(M)\leq N$ and the Ricci curvature is bounded below by $K$ (resp. if and only if the $N$-Bakry-\'Emery-Ricci tensor is bounded below by $K$).
\item  Define the \emph{slope} of  a real valued function $u:X\to \R$ at the point $x\in X$ as
	\begin{equation*}
	|\nabla u|(x):= 
		\begin{cases}
			\limsup_{y\to x} \frac{|u(x)-u(y)|}{\di(x,y)} 	&\text{if $x$ is not isolated}		\\
			0		&\text{otherwise}.
		\end{cases}
	\end{equation*}
We denote with ${\rm LIP} (X)$ the space of Lipschitz functions on $(X,\di)$.
\item  Let $f\in L^2(X, \mu)$. The \emph{Cheeger energy} of $f$ is defined as
	\begin{equation*}
	\mathsf{Ch}(f):= \inf \left\{\liminf_{n\to \infty}\frac{1}{2}\int |\nabla f_n|^2  \dd\mu \, |\, 	f_n\in {\rm LIP}(X)\cap L^2(X,\mu), \|f_n-f\|_{L^2}\to 0 \right\}.
	\end{equation*}
	One can check that the Cheeger energy $\mathsf{Ch}:L^2(X,\mu)\to [0,\infty]$ is convex and lower semi-continuous. Thus it admits an $L^2$-gradient flow, called \emph{heat flow}.
	\item The metric measure space $(X,\di,\mu)$ is said \emph{infinitesimally Hilbertian} if  $\mathsf{Ch}$ is a quadratic form, i.e. it satisfies the parallelogram identity.
	\\One can check that $(X,\di,\mu)$ is infinitesimally Hilbertian if and only if the heat flow for every positive time is a linear map from $L^2(X,\mu)$ to $L^2(X,\mu)$.
	\\If $(X,\di,\mu)$ is the metric measure space associated to a smooth Finsler manifold, one can check that  $(X,\di,\mu)$ is  infinitesimally Hilbertian if and only if the manifold is actually Riemannian.
\item Given $K\in \R$ and $N\in (1,\infty]$,  we say that $(X,\di,\mu)$ verifies the \emph{$\mathsf{RCD}(K,N)$ condition} if it satisfies the $\mathsf{CD}(K,N)$ condition and it is infinitesimally Hilbertian.	
\item \emph{Pointed measured Gromov-Hausdorff convergence}: Let $(X_n,\di_n,\mu_n)$, $n\in \mathbb{N}\cup \{\infty\}$, be  a sequence of  metric measure spaces and let $\bar{x}_n\in X_n$ for every $n\in \mathbb{N}\cup \{\infty\}$ be a sequence of reference points. 
We say that $(X_n,\di_n,\mu_n, \bar x_n)\to
(X_\infty,\di_\infty,\mu_\infty, \bar x_\infty)$ in the 
pointed measured Gromov Hausdorff (pmGH) sense, provided for any $\varepsilon,R>0$ there exists $N({\varepsilon,R})\in \mathbb{N}$ such that for all $n\geq N({\varepsilon,R})$ there exists a Borel map $f^{R,\varepsilon}_n:B_R(\bar x_n)\to X_\infty$ such that
\begin{itemize}
\item $f^{R,\varepsilon}_n(\bar x_n)=\bar x_\infty$,
\item $\sup_{x,y\in B_R(\bar x_n)}|\di_n(x,y)-\di_\infty(f^{R,\varepsilon}_n(x),f^{R,\varepsilon}_n(y))|\leq\varepsilon$,
\item the $\varepsilon$-neighbourhood of $f^{R,\varepsilon}_n(B_R(\bar x_n))$ contains $B_{R-\varepsilon}(\bar x_\infty)$,
\item  $(f^{R,\varepsilon}_n)_\sharp(\mu_n\llcorner{B_R(\bar x_n)})$ weakly converges to $\mu_\infty\llcorner{B_R(x_\infty)}$ as $n\to\infty$, for a.e. $R>0$.
\end{itemize}
If in addition there exists $\bar{R}>0$ such that ${\rm diam}(X_n)\leq \bar{R}$ for every  $n\in \mathbb{N}\cup \{\infty\}$, then we say that $(X_n,\di_n,\mu_n)\to
(X_\infty,\di_\infty,\mu_\infty)$ in the  \emph{measured Gromov Hausdorff (mGH for short)} sense. In this case it is enough to consider only $R=\bar{R}$ in the above requirements.

\item  \emph{Stability}:  Let $K\in \R$ and $N\in (1,\infty]$ be given.  Assume that $(X_n, \di_n, \mu_n)$ satisfies  $\mathsf{CD}(K,N)$ (resp. $\mathsf{RCD}(K,N)$), for every $n\in \mathbb{N}$, and that  $(X_n, \di_n, \mu_n, \bar{x}_n )\to  (X_\infty,\di_\infty,\mu_\infty, \bar x_\infty)$ in the pmGH sense. Then 	$(X_\infty, \di_\infty, \mu_\infty)$ satisfies 	 $\mathsf{CD}(K,N)$   (resp. $\mathsf{RCD}(K,N)$)  as well.
\end{itemize}

\subsection{Entropy functionals}
In this section we assume that $X$ is a Polish space.

A function $F:[0,+\infty)\rightarrow [0,+\infty]$ belongs to the class $\Gamma_0({\mathbb{R}_{+}})$ of the \emph{admissible entropy functions} if $F$ is convex, lower semicontinuous and $F(1)=0$.
We define the recession constant as
$$F'_{\infty}=\lim_{s \to \infty}\frac{F(s)}{s},$$
and we say that $F$ is superlinear if $F'_{\infty}=+\infty$.

We also define the perspective function induced by $F\in \Gamma_0(\mathbb{R}_+)$ as the function $\hat{F}:[0,+\infty)\times [0,+\infty)\rightarrow [0,+\infty]$, given by
\begin{equation}\label{def: funzione prospettiva}
\hat{F}(r,t):=\begin{cases} F\big(\frac{r}{t}\big)t &\mbox{if} \ t>0,\\
F'_{\infty}r &\mbox{if} \ t=0.\end{cases}
\end{equation}
The function 
\begin{equation}\label{def: reverse entropy}
R:[0,+\infty)\rightarrow [0,+\infty], \qquad R(t):=\hat{F}(1,t)
\end{equation}
is called \emph{reverse entropy}.

Let $F\in \Gamma_{0}(\mathbb{R}_{+})$ be an admissible entropy function. The \emph{$F$-divergence} (also called \emph{Csisz\'ar's divergence} or \emph{relative entropy}) is the functional $D_F:\mathcal{M}(X)\times \mathcal{M}(X)\rightarrow [0,+\infty]$ defined by
\begin{equation}\label{def: F divergenza}
D_F(\gamma||\mu):=\int_X F(\sigma)\mathrm{d} \mu + F'_{\infty}\gamma^{\perp}(X), \qquad \gamma=\sigma\mu+\gamma^{\perp},
\end{equation}
where $\gamma=\sigma\mu+\gamma^{\perp}$ is the Lebesgue's decomposition of the measure $\gamma$ with respect to $\mu$.
When $F$ is superlinear $D_F(\gamma||\mu)=+\infty$ if $\gamma$ has a singular part with respect to $\mu$. Moreover, it is clear that $D_F(\mu||\mu)=0$. 

We now collect some useful properties of the relative entropies. For the proof see \cite[Section 2.4]{LMS}.

\begin{lemma}\label{lem: proprieta funzionale entropia}
The functional $D_F$ is jointly convex and lower semicontinuous in $\Mi(X)\times \Mi(X)$. More generally, if $F\in \Gamma_0(\R _+)$ is the pointwise limit of an increasing sequence $(F_n)\subset \Gamma_0(\R _+)$ and $\gamma,\mu\in \Mi(X)$ are the weak limit of a sequence $(\gamma_n,\mu_n)\subset \Mi(X)\times \Mi(X)$ then we have
$$\liminf D_{F_n}(\gamma_n||\mu_n)\geq D_F(\gamma||\mu).$$
\end{lemma}

\begin{lemma}\label{lem: compattezza sottolivelli entropia}
If $\mathcal{K}\subset \Mi(X)$ is bounded and $F'_{\infty}>0$ then the set 
\begin{equation}
\mathbf{K}_C:=\{\gamma\in \Mi(X): D_F(\gamma||\mu)\leq C, \ \textrm{for some} \ \mu\in \mathcal{K} \} 
\end{equation}
is bounded for every $C\geq 0$. Moreover, if $\mathcal{K}$ is also equally tight and $F$ is superlinear, then $\mathbf{K}_C$ is equally tight for every $C\geq 0$.
\end{lemma}

The last lemma of this section shows an invariance result for the $F$-divergences.

\begin{lemma}\label{lem: invariance divergence under injection}
Let $F\in \Gamma_0(\mathbb{R_{+}})$ be an admissible entropy function, $X,Y$ be two Polish spaces and $f:X\rightarrow Y$ be a Borel injective map.  Then, for any $\gamma,\mu\in \Mi(X)$ it holds
\begin{equation}
D_F(\gamma||\mu)=D_F(f_{\sharp}\gamma||f_{\sharp}\mu).
\end{equation}
\end{lemma}
\begin{proof}
Let us consider the Lebesgue's decompositions 
$$\gamma=\sigma \mu + \gamma^{\perp} \qquad\textrm{and} \qquad f_{\sharp}\gamma=\tilde{\sigma}f_{\sharp}\mu+\tilde{\gamma}^{\perp}.$$
Since $f_{\sharp}\gamma$ and $f_{\sharp}\mu$ have support contained in $f(X)$, we can suppose without loss of generality that $f$ is bijective. 
\\For any Borel set $A\subset X$ we have
\begin{multline}
\int_A\sigma\,\dd \mu +\gamma^{\perp}(A)=\gamma(A)=\gamma(f^{-1}(f(A)))=f_{\sharp}\gamma(f(A))\\
=\int_{f(A)} \tilde{\sigma}\,\dd f_{\sharp}\mu+\tilde{\gamma}^{\perp}(f(A))=\int_A  \tilde{\sigma}\circ f\,\dd \mu + \tilde{\gamma}^{\perp}(f(A)).
\end{multline}
By the uniqueness of the Lebesgue's decomposition (see \cite[Lemma 2.3]{LMS}) it follows that $\sigma=\tilde{\sigma}\circ f$ up to $(\mu + \gamma)$-negligible sets and $\gamma^{\perp}(X)=\tilde{\gamma}^{\perp}(f(X))=\tilde{\gamma}^{\perp}(Y)$.
In particular
$$D_F(f_{\sharp}\gamma||f_{\sharp}\mu):=\int_Y F(\tilde{\sigma})\mathrm{d} f_{\sharp}\mu + F'_{\infty}\tilde{\gamma}^{\perp}(Y)=\int_X F(\tilde{\sigma}\circ f)\dd \mu + F'_{\infty}\gamma^{\perp}(X)=D_F(\gamma||\mu).$$
\end{proof}

\section{Entropy-Transport problem and distances}

\noindent
Let $\boldsymbol{\gamma}\in \Mi(X\times X)$. In the sequel we denote by $\gamma_i:=(\pi^i)_{\sharp}\boldsymbol{\gamma}$ the marginals of $\boldsymbol{\gamma}$. 
\\We are now ready to define the Entropy-Transport problem.

\begin{definition}\label{def: primal formulation ET}
Let $F\in \Gamma_{0}(\mathbb{R}_{+})$ and let $\boldsymbol{\mathrm c}:X\times X\rightarrow [0,+\infty]$ be a lower semicontinuous function. The Entropy-Transport functional between the measures $\mu_1,\mu_2\in \mathscr{M}(X)$ is the functional 
\begin{equation}\label{def: funzionale di trasporto entropico}
\begin{aligned}
&\mathcal{ET}(\,\cdot \,||\mu_1,\mu_2):\mathscr{M}(X\times X)\rightarrow [0,+\infty], \\
&\mathcal{ET}(\boldsymbol{\gamma}||\mu_1,\mu_2):=D_F(\gamma_1||\mu_1) + D_F(\gamma_2||\mu_2) + \int_{X\times X}\boldsymbol{\mathrm c}(x_1,x_2)\dd \boldsymbol{\gamma}(x_1,x_2).
\end{aligned}
\end{equation} 
\smallskip
We define the Entropy-Transport problem between $\mu_1$ and $\mu_2$ as the minimization problem
\begin{equation}\label{def: problema di trasporto entropico}
\et(\mu_1,\mu_2):=\inf_{\boldsymbol{\gamma}\in \Mi(X\times X)}\mathcal{ET}(\boldsymbol{\gamma}||\mu_1,\mu_2).
\end{equation}
To highlight the role of the entropy function $F$ and the cost function $\boldsymbol{\mathrm c}$, we also say that $\et$ is the cost of the Entropy-Transport problem induced by $(F,c)$.
\end{definition}

We are particularly interested in cost functions of the form $\boldsymbol{\mathrm c}(x_1,x_2)=\ell(\di(x_1,x_2))$ for a certain function $\ell:[0,\infty)\rightarrow [0,\infty]$.

In the next Proposition we recall some properties of Entropy-Transport problems (for a proof see \cite{LMS}).
\begin{proposition}\label{prop: proprieta ET}
Let us suppose that the Entropy-Transport problem between the measures $\mu_1,\mu_2\in \mathscr{M}(X)$ is feasible, i.e. there exists $\boldsymbol{\gamma}\in \Mi(X\times X)$ such that $\mathcal{ET}(\boldsymbol{\gamma}||\mu_1,\mu_2)<\infty$, and that $F$ is superlinear. Then the infimum in \eqref{def: problema di trasporto entropico} can be replaced by a minimum and the set of minimizers is a compact convex subset of $\Mi(X\times X)$. Moreover, the functional $\et$ is convex and positively $1$-homogeneous (thus subadditive). 
\end{proposition}

\begin{remark}\label{rem: conical formulation}
An important role in the theory of Entropy-Transport problems is played by the \emph{marginal perspective cost} $H$, that we are going to define.

Given a number $c\in[0,+\infty)$ and an admissible entropy function $F$, we first introduce the marginal perspective function $H_c:[0,+\infty)\times[0,+\infty)\rightarrow [0,+\infty]$ as the lower semicontinuous envelope of the function
\begin{equation*}
\label{espressione H tilde}
\tilde{H}_c(r_1,r_2):=\inf_{\theta>0}R\left(\frac{r_1}{\theta}\right)\theta+R\left(\frac{r_1}{\theta}\right)\theta+\theta c,
\end{equation*}
where $R$ is the reverse entropy defined in \eqref{def: reverse entropy}.
If $c=+\infty$, we set 
\begin{equation*}
H_{\infty}(r_1,r_2)=F(0)r_1+F(0)r_2.
\end{equation*}
When $\boldsymbol{\mathrm c}:X_1\times X_2\rightarrow [0,+\infty]$ is a lower semicontinuous cost function on two metric spaces $X_1,X_2$, the induced marginal perspective cost
$$H:X_1\times [0,+\infty)\times X_2\times[0,+\infty)\rightarrow [0,+\infty]$$
is defined as
\begin{equation}\label{eq:defHHc}
H(x_1,r;x_2,t):=H_{\boldsymbol{\mathrm c}(x_1,x_2)}(r,t).
\end{equation}

One can give some equivalent formulations of the problem \eqref{def: problema di trasporto entropico} in terms of the marginal perspective cost (see for instance \cite[Theorem 5.8]{LMS}). Moreover, the metric properties of the entropy-transport cost $\et$ defined in \eqref{def: problema di trasporto entropico} can be read in terms of the properties of $H$, studied as a function on the space $\mathfrak{C}(X)\times \mathfrak{C}(X)$. This point of view, which links the Entropy-Transport structure with the conical geometry of the problem, has been deeply investigated by Liero, Mielke and Savaré for the Hellinger-Kantorovich distance \cite[Section 7]{LMS}  (see also \cite{Chizat, DepIeee} and \cite[Chapters 3,4]{DepPhd} for general marginal perspective functions). 

For brevity, we do not enter into the details of all these formulations (but see Section \ref{sec: CGW} for some details on the conical construction performed in \cite{SVP}). Here we only remark that for any complete and separable metric space $(X,\di)$ the cost $\et^a$ induces a distance on the space of measures $\mathscr{M}(X)$ if and only if $H^a$ is a distance on the cone $\mathfrak{C}(X)$, $a\in (0,1]$. In general, it is not difficult to identify conditions on $F$ and $\boldsymbol{\mathrm c}$ for which the induced function $H$ is nonnegative, symmetric and $H(x_1,r;x_2,t)=0$ if and only if $(x_1,r)=(x_2,t)$ as points on the cone (see \cite[Proposition 4]{DepIeee}); on the contrary, proving the triangle inequality for (a power of) $H$ is a much more challenging problem.
\end{remark}

\subsection{Regular Entropy-Transport distances}
In the next definition we introduced the class of \emph{regular Entropy-Transport distances}.

\begin{definition}\label{def: Entropy-Transport distance}
We say that $\Det$ is a regular Entropy-Transport distance if 
\begin{itemize}
\item There exist $a\in (0,1]$, $F\in \Gamma_0(\R_+)$ and  a function $\ell:[0,\infty)\rightarrow [0,\infty]$ such that for every complete and separable metric space $(X,\di)$, setting $\boldsymbol{\mathrm c}(x_1,x_2):=\ell(\di(x_1,x_2))$, the function $\Det$ coincides with the power $a$ of the Entropy-Transport cost $\et$ induced by $(F,\boldsymbol{\mathrm c})$, namely
\begin{equation}
\Det(\mu_1,\mu_2)=\et^a(\mu_1,\mu_2) \quad \textrm{for every} \ \mu_1,\mu_2\in \Mi(X).
\end{equation}
\item The function $\ell$ is continuous, convex and $\ell(s)=0$ if and only if $s=0$. 
\item $F$ is superlinear and finite valued. 
\item  For every complete and separable metric space $(X,\di)$, the related Entropy-Transport distance $\Det$ is a complete and separable metric on $\Mi(X)$ inducing the weak topology.
\end{itemize}
We also write that the distance $\Det$ is induced by $(a,F,\ell)$ with obvious meaning.
\end{definition}

We notice that if $\Det$ is a regular Entropy-Transport distance induced by $(a,F,\ell)$ then $\ell$ is an increasing function and $\lim_{d\to +\infty}\ell(d)=+\infty$.

We conclude the section with a list of examples of regular Entropy-Transport distances.

\begin{examples*}
\begin{enumerate}
\item\label{def:HK} \textbf{Hellinger-Kantorovich:} Let $F(s)=U_1(s):=s\log{s}-s+1$ and $$\ell_{\hk}(d):=\begin{cases} -\log\left({\cos^2(d)}\right) \ \ &\textrm{if} \ d<\frac{\pi}{2}, \\
+\infty \ \ &\textrm{otherwise}.\end{cases}$$
\\It is proved in \cite[Section 7]{LMS} that $(1/2,U_1,\ell_{\hk})$ induces a regular Entropy-Transport distance, called \emph{Hellinger-Kantorovich} distance. We refer also to \cite{LMS1} for a discussion on ``weighted versions'' of the Hellinger-Kantorovich distance.

\item\label{def:GHK} \textbf{Gaussian Hellinger-Kantorovich:} Let $F(s)=U_1(s)=s\log{s}-s+1$ and $\ell_2(d):=d^2$. 
\\ The triple $(1/2,U_1,\ell_2)$ induces a regular Entropy-Transport distance, as discussed in \cite[Section 7.8]{LMS}. It is called \emph{Gaussian Hellinger-Kantorovich} distance.

\item\label{def:power-like} \textbf{Quadratic power-like distances:} Let 
$$F(s)=U_p(s):=\frac{s^p-p(s-1)-1}{p(p-1)}, \qquad p>1$$ 
and $\ell_2(d)=d^2$.
\\Then, for every $1<p\le 3$ the triple $(1/2,U_p,\ell_2)$ induces a regular Entropy-Transport distance, as proved in \cite[Theorem 6 and Corollary 1]{DepIeee}.
\\ We notice that the class of entropy functions $\{U_p\}$ satisfies
$\lim_{p\to 1} U_p(s)=U_1(s)$, justifying the notation we have used (see also \cite[Example 2.5]{LMS}).

\item\label{def:linear power-like} \textbf{Linear power-like distances:}
Let $$F(s)=U_p(s):=\frac{s^p-p(s-1)-1}{p(p-1)}, \qquad p>1$$
and $\ell_1(d):=d$. 
\\ For every $p>1$, $(1/2,U_p,\ell_1)$ induces a regular Entropy-Transport distance (see again \cite[Theorem 6 and Corollary 1]{DepIeee}). 

\end{enumerate}
\end{examples*}

\section{Sturm-Entropy-Transport distance}

We say that two metric measure spaces $(X_1,\mathsf{d}_1,\mu_1)$ and $(X_2,\mathsf{d}_2,\mu_2)$ are \emph{isomorphic} if there exists an isometry $\psi:\mathsf{supp}(\mu_1)\rightarrow \mathsf{supp}(\mu_2)$ such that $\psi_{\sharp}\,\mu_1=\mu_2$, where $\psi_{\sharp}$ denotes the push-forward through the map $\psi$.
A necessary condition in order to be isomorphic is that $\mu_1(X_1)=\mu_2(X_2).$
\\The family of all isomorphism classes of metric measure spaces will be denoted by $\boldsymbol{\mathrm{X}}$.
From now on, we will identify a metric measure space with its class.
\\
We recall now the definition of the $\mathbf{D}_p$-distance due to Sturm.
\begin{definition}[\cite{St1}]\label{def: D}
Fix $p\ge 1$. Let $(X_1,\mathsf{d}_1,\mu_1)$ and $(X_2,\mathsf{d}_2,\mu_2)$ be two metric measure spaces, the Sturm $\mathbf{D}_p$-distance is defined as 
\begin{equation}
\mathbf{D}_p\big((X_1,\mathsf{d}_1,\mu_1),(X_2,\mathsf{d}_2,\mu_2)\big):=\inf \mathcal{W}_p(\psi^1_{\sharp}\mu_1,\psi^2_{\sharp}\mu_2),
\end{equation}
where the infimum is taken over all complete and separable metric spaces $(\hat{X},\hat{\mathsf{d}})$ with isometric embeddings $\psi^1:\mathsf{supp}(\mu_1)\rightarrow \hat{X}$ and $\psi^2:\mathsf{supp}(\mu_2)\rightarrow \hat{X}$.
\end{definition}
It is proved in \cite[Theorem 3.6]{St1} that $\mathbf{D}_p$ is a complete, separable and geodesic metric on the set
$$\boldsymbol{\mathrm{X}}_{1,p}:=\{ (X,\mathsf{d},\mu)\in \boldsymbol{\mathrm{X}}: \ \mu\in \mathscr{P}_p(X,\mathsf{d})\}.$$

We are now going to define the \emph{Sturm-Entropy-Transport distance} in a similar way.

\begin{definition}\label{def: D-ET}
Let $(X_1,\mathsf{d}_1,\mu_1)$ and $(X_2,\mathsf{d}_2,\mu_2)$ be two metric measure spaces, we define the Sturm-Entropy-Transport distance induced by the regular Entropy-Transport distance $\Det$ as
\begin{equation}
\DET\big((X_1,\mathsf{d}_1,\mu_1),(X_2,\mathsf{d}_2,\mu_2)\big):=\inf \Det (\psi^1_{\sharp}\mu_1,\psi^2_{\sharp}\mu_2),
\end{equation}
where the infimum is taken over all complete and separable metric spaces $(\hat{X},\hat{\mathsf{d}})$ with isometric embeddings $\psi^1:\mathsf{supp}(\mu_1)\rightarrow \hat{X}$ and $\psi^2:\mathsf{supp}(\mu_2)\rightarrow \hat{X}$.
\end{definition}
It is not difficult to prove that the definition is well-posed. Indeed, let us suppose $(X'_i,\mathsf{d}'_i,\mu'_i)$ is isomorphic to $(X_i,\mathsf{d}_i,\mu_i)$ through the map $\varphi^i$, $i=1,2$. Then, for every metric space $\hat{X}$ and every isometric embedding $\psi^i:\mathsf{supp}(\mu_i)\rightarrow \hat{X}$, $i=1,2$, we have that 
$$\Det \left((\psi^1\circ\varphi^1)_{\sharp}\mu_1,(\psi^2\circ\varphi^2)_{\sharp}\mu_2\right)=\Det (\psi^1_{\sharp}\mu_1,\psi^2_{\sharp}\mu_2).$$

It is often convenient to work with explicit realisations of the  ambient space  $(\hat{X},\hat{\mathsf{d}})$, a particularly useful one is given by the disjoint union that we now discuss. 
\\Given two metric spaces $(X_1,\mathsf{d}_1,\mu_1)$ and $(X_2,\mathsf{d}_2,\mu_2)$,
let $X_1\sqcup X_2$ be their disjoint union. We say that a (resp. pseudo-)metric $\hat{\mathsf{d}}$ on $X_1\sqcup X_2$ is a \emph{(resp. pseudo-)metric coupling} between $\mathsf{d}_1$ and $\mathsf{d}_2$ if $\hat{\mathsf{d}}(x,y)=\mathsf{d_1}(x,y)$ when $x,y\in X_1$ and $\hat{\mathsf{d}}(x,y)=\mathsf{d_2}(x,y)$ when $x,y\in X_2$.
\\
A finite valued metric coupling $\hat{\mathsf{d}}$ between $\mathsf{d}_1$ and $\mathsf{d}_2$ always exists: to construct it, fix two points $\bar{x}_1\in X_1, \bar{x}_2\in X_2$, a number $c\in \mathbb{R}_+$, and define $\hat{\mathsf{d}}$ as:
\begin{equation}
\hat{\mathsf{d}}(x,y):=\begin{cases} \mathsf{d}_1(x,y) \ \ &\textrm{if} \ x,y\in X_1\\
\mathsf{d}_2(x,y) &\textrm{if} \ x,y\in X_2\\
\mathsf{d}_1(x,\bar{x}_1)+c+\mathsf{d}_2(\bar{x}_2,y) \ \  &\textrm{if} \ x\in X_1, y\in X_2\\
\mathsf{d}_1(y,\bar{x}_1)+c+\mathsf{d}_2(\bar{x}_2,x) \ \  &\textrm{if} \ y\in X_1, x\in X_2.\end{cases}
\end{equation}
Moreover, from any finite valued pseudo-metric coupling $\hat{\mathsf{d}}$ of $\mathsf{d}_1$ and $\mathsf{d}_2$ and any $\delta>0$ we can obtain a complete, separable metric $\hat{\mathsf{d}}_{\delta}$ which is again a coupling of $\mathsf{d}_1$ and $\mathsf{d}_2$ in the following way:
\begin{equation}\label{def: metrica coupling}
\hat{\mathsf{d}}_{\delta}:=\begin{cases} \hat{\mathsf{d}} \ \ &\textrm{on} \ (X_1\times X_1)\sqcup (X_2\times X_2)\\
\hat{\mathsf{d}}+\delta &\textrm{on} \ (X_1\times X_2)\sqcup (X_2\times X_1).\end{cases}
\end{equation}
We say that a measure $\boldsymbol{\gamma}\in \mathscr{M}(X_1\times X_2)$ is a \emph{measure coupling} between $\mu_1$ and $\mu_2$ if 
\begin{equation}
\boldsymbol{\gamma}(A\times X_2)=\mu_1(A) \ \ \textrm{and} \ \ \ \boldsymbol{\gamma}(X_1\times B)=\mu_2(B), 
\end{equation}
for all Borel sets $A\subset X_1$ and $B\subset X_2$.
We keep the notation $\gamma_i$ for the marginals of the measure $\boldsymbol{\gamma}\in \mathscr{M}(X_1\times X_2)$, $i=1,2.$

A more explicit formulation of the function $\DET$ is given in the following Proposition. 
\begin{proposition}\label{pr: formulazione D-LET}
Let $(X_1,\mathsf{d}_1,\mu_1)$ and $(X_2,\mathsf{d}_2,\mu_2)$ be two metric measure spaces and $\Det$ a regular Entropy-Transport distance induced by $(a,F,\ell)$. 
\begin{enumerate}
\item[{\rm(i)}] In Definition \ref{def: D-ET} we can suppose without loss of generality that $\hat{X}=X_1\sqcup X_2$, $\psi^1=\iota_1$, $\psi^2=\iota_2$ be respectively the inclusion of $X_1$ and $X_2$ in $X_1\sqcup X_2$ and the infimum is taken over all the pseudo-metric couplings $\hat{\mathsf{d}}$ between $\mathsf{d}_1$ and $\mathsf{d}_2$.
\item[{\rm(ii)}]  In the situation of {\rm(i)}  we will identify $\mu_k$ with $(\iota_k)_\sharp \,\mu_k$, $k=1,2$, and it holds 
\begin{equation}\label{eq:DETainfC}
\DET^{1/a}((X_1,\mathsf{d}_1,\mu_1),(X_2,\mathsf{d}_2,\mu_2))=\inf_C\left\{\sum_{i=1}^2D_F(\gamma_i||\mu_i)+\int_{X_1\times X_2}\ell\big(\hat{\mathsf{d}}(x,y)\big)\dd\boldsymbol{\gamma}\right\},
\end{equation}
where 
\begin{multline}
C:=\{(\boldsymbol{\gamma},\hat{\di}): \boldsymbol{\gamma}\in \mathscr{M}(X_1\times X_2), \ \hat{\di} \ \textrm{finite valued pseudo-metric coupling for} \ \di_1,\di_2\}
\end{multline}

\end{enumerate}
\end{proposition}
\begin{proof}
${\rm(i)}$ We first show that the infimum as in  ${\rm(i)}$ is less or equal to the infimum as in Definition \ref{def: D-ET}. Let $(\hat{X},\hat{\mathsf{d}})$ be a complete and separable metric space with isometric embeddings $\psi^1:\mathsf{supp}(\mu_1)\rightarrow \hat{X}$, $\psi^2:\mathsf{supp}(\mu_2)\rightarrow \hat{X}$, and let  $\hat{\boldsymbol{\gamma}}\in \mathscr{M}(\hat{X}\times \hat{X})$.
It is immediate to check that    
\begin{equation}
\tilde{{\mathsf{d}}}(x_{1},x_{2}):=
\begin{cases} 
\di_1(x_1,x_2) &  \text{ if } (x_{1},x_{2})\in X_1\times X_1\\ 
\di_2(x_1,x_2) &  \text{ if } (x_{1},x_{2})\in X_2\times X_2\\
\displaystyle \inf_{\substack{y_1\in \mathsf{supp}(\mu_1)\\ y_2 \in \mathsf{supp}(\mu_2)}} \di_1(x_1,y_1)+\hat{\di}(\psi^1(y_1),\psi^2(y_2))+\di_2(y_2,x_2)&  \text{ if } (x_{1},x_{2})\in X_1\times X_2\\
\displaystyle \inf_{\substack{y_1\in \mathsf{supp}(\mu_1)\\ y_2 \in \mathsf{supp}(\mu_2)}} \di_1(x_2,y_1)+\hat{\di}(\psi^1(y_1),\psi^2(y_2))+\di_2(y_2,x_1)&  \text{ if } (x_{1},x_{2})\in X_2\times X_1
\end{cases}
\end{equation}
defines a pseudo-metric on $X_1\sqcup X_2$, coupling between $\mathsf{d}_1$ and $\mathsf{d}_2$. \\ Moreover, setting the Borel injective functions $\Psi^i: \psi^1(\mathsf{supp}(\mu_1))\cup  \psi^2(\mathsf{supp}(\mu_2))\subset \hat{X} \to X_1\sqcup X_2$, $i=1,2$, defined as 
\begin{equation*}
\Psi^1(\hat{x}):=
\begin{cases}
\iota_{1} ((\psi^1)^{-1}(\hat{x}))& \text{ if } \hat{x}\in \psi^1(\mathsf{supp}(\mu_1))  \\
\iota_{2} ((\psi^2)^{-1}(\hat{x}))& \text{ if } \hat{x}\in \psi^2(\mathsf{supp}(\mu_2)),  \hat{x} \notin \psi^{1}(\mathsf{supp}(\mu_1)),
\end{cases}
\end{equation*}
\begin{equation*}
\Psi^2(\hat{x}):=
\begin{cases}
\iota_{1} ((\psi^1)^{-1}(\hat{x}))& \text{ if } \hat{x}\in \psi^1(\mathsf{supp}(\mu_2)),  \hat{x} \notin \psi^{2}(\mathsf{supp}(\mu_1))  \\
\iota_{2} ((\psi^2)^{-1}(\hat{x}))& \text{ if } \hat{x}\in \psi^2(\mathsf{supp}(\mu_2)),
\end{cases}
\end{equation*}
and using Lemma \ref{lem: invariance divergence under injection}  it is immediate to check that $\tilde{\boldsymbol{\gamma}}:= (\Psi^1,\Psi^2)_{\sharp} \hat{\boldsymbol{\gamma}}\in \mathscr{M}((X_1\sqcup X_2)\times (X_1\sqcup X_2))$ satisfies  
\begin{align}
&\sum_{i=1}^2D_F(\tilde{\gamma}_i||(\iota_{i})_{\sharp}\mu_i)+\int_{\iota_{1}(X_1)\times \iota_{2}(X_2)}\ell\big(\tilde{\mathsf{d}}(x,y)\big)\dd\boldsymbol{\tilde{\gamma}}(x,y) \nonumber\\
& \quad \le \sum_{i=1}^2D_F(\hat{\gamma}_i||(\psi^{i})_{\sharp}\mu_i)+\int_{\psi^{1}(\mathsf{supp}(\mu_1))\times \psi^{2}(\mathsf{supp}(\mu_2))}\ell\big(\hat{\mathsf{d}}(x,y)\big)\dd\boldsymbol{\hat{\gamma}}(x,y) \nonumber\\
& \quad \leq \sum_{i=1}^2D_F(\hat{\gamma}_i||(\psi^{i})_{\sharp}\mu_i)+\int_{\hat{X}\times \hat{X}}\ell\big(\hat{\mathsf{d}}(x,y)\big)\dd\boldsymbol{\hat{\gamma}}(x,y), \label{eq:tildedhatdgeq}
\end{align}
where we have used the fact that $\tilde{\mathsf{d}}(\Psi^1(x),\Psi^2(y))\le \hat{\mathsf{d}}(x,y)$ whenever $x\in \psi^1(\mathsf{supp}(\mu_1))$ and $y\in \psi^2(\mathsf{supp}(\mu_2)).$

This yields that the infimum as in  ${\rm(i)}$ is less or equal to the infimum as in Definition \ref{def: D-ET}.

To show that the infimum as in Definition \ref{def: D-ET} is less or equal to the infimum as in  ${\rm(i)}$, it is sufficient to notice that for every pseudo-metric coupling $\hat{\di}$, for every measure $\boldsymbol{\gamma}\in \mathscr{M}(X_1\times X_2)$ and for every $\epsilon>0$ there is $\delta>0$ such that  the complete and separable metric $\hat{\mathsf{d}}_{\delta}$ defined in \eqref{def: metrica coupling} is a coupling between $\di_1$ and $\di_2$ satisfying
\begin{equation}\label{eq:ddeltadeps}
\int_{X_1\times X_2}\ell\big(\hat{\mathsf{d}}_{\delta}(x,y)\big)\dd\boldsymbol{\gamma} \le \int_{X_1\times X_2}\ell\big(\hat{\mathsf{d}}(x,y)\big)\dd\boldsymbol{\gamma}+\epsilon,
\end{equation}
as a consequence of the finiteness of the measure $\boldsymbol{\gamma}$ and the continuity of $\ell$. 
\vspace{0.3cm}

${\rm(ii)}$ In case the infimum runs over the couples $(\boldsymbol{\gamma},\hat{\di})\in C$ such that $\hat{\di}$ is a complete and separable metric, the inequality ``$\leq$'' in \eqref{eq:DETainfC}  is a simple consequence of the explicit formulation of the Entropy-Transport problem together with the fact that the superlinearity of $F$ allows to consider measures $\boldsymbol{\gamma}\in\mathscr{M}((X_1\sqcup X_2)\times (X_1\sqcup X_2))$ with support contained in $X_1\times X_2$. The fact that ``$\leq$'' holds  in \eqref{eq:DETainfC} even if the infimum is taken over the larger set $C$ is a consequence of \eqref{eq:ddeltadeps}.

The proof of the inequality  ``$\geq$'' in \eqref{eq:DETainfC} is analogous to the first part of the proof of  ${\rm(i)}$, see in particular \eqref{eq:tildedhatdgeq}.
\end{proof}

In the next Lemma we collect some of the basic properties of the function $\DET.$ 
\begin{lemma}\label{lem: proprieta' di base}
Let $\Det$ be a regular Entropy-Transport distance induced by $(a,F,\ell)$. 
\begin{enumerate}[(i)]
\item\label{lem: riscalamento misura}
For any $M\ge 0$ it holds 
\begin{equation}
\DET((X_1,\mathsf{d}_1,M\mu_1),(X_2,\mathsf{d}_2,M\mu_2))=M^a\DET((X_1,\mathsf{d}_1,\mu_1),(X_2,\mathsf{d}_2,\mu_2)).
\end{equation}

\item\label{lem: dis metriche coincidenti}
If $(X_1,\mathsf{d}_1)=(X_2,\mathsf{d}_2)$ then
\begin{equation}
\DET((X_1,\mathsf{d}_1,\mu_1),(X_2,\mathsf{d}_2,\mu_2))\leq \Det(\mu_1,\mu_2).
\end{equation}

\item\label{lem: densita' delta}
The set 
\begin{equation}\label{eq: densita' delta}
\boldsymbol{\mathrm{X}}_{*}:=\Big\{(X,\mathsf{d},\mu)\in\boldsymbol{\mathrm{X}}, \ \mathsf{supp}(\mu)=\{x_1,...,x_n\}, \ n \in \mathbb{N}, \ \mu=M\sum_{i=1}^n \delta_{x_i}, \ M\in \mathbb{R}_{+}\Big\}
\end{equation}
is dense in $(\boldsymbol{\mathrm{X}},\DET).$

\item\label{lem: dis tra delta}
If 
\begin{equation}
\mu=M\sum_{i=1}^n \delta_{x_i} \ \textrm{and} \ \mu'=M\sum_{i=1}^n \delta_{x'_i},
\end{equation}
then
\begin{equation}
\DET^{1/a}((X,\mathsf{d},\mu),(X',\mathsf{d}',\mu'))\leq Mn\,\ell\Big(\sup_{i,j}|\mathsf{d}_{ij}-\mathsf{d}'_{ij}|\Big),
\end{equation}
where we put $\mathsf{d}_{ij}=\mathsf{d}(x_i,x_j)$ and $\mathsf{d}'_{ij}=\mathsf{d}(x'_i,x'_j)$.

\item\label{lem: dis riscalamento misura}
For any $N>1$ there exists a constant $C$ such that for every $M$, $1/N<M<N$, we have
\begin{equation}
\DET^{1/a}((X,\mathsf{d},\mu),(X,\mathsf{d},M\mu))\leq C\mu(X)|M-1|.
\end{equation}
\end{enumerate}
\end{lemma}
\begin{proof}
\begin{enumerate}[(i)]
\item This is a consequence of the $1$-homogeneity of the cost $\et$ (Proposition \ref{prop: proprieta ET}) and of the push-forward map together with the definitions of $\Det$ and $\DET$. 
\vspace{0.3cm}

\item The result follows from the definition of $\DET$, since $(\hat{X},\hat{\di})=(X_1,\di_1)$ with $\psi_1=\psi_2=\mathrm{Id}$ is an admissible competitor for the infimum.
\vspace{0.3cm}

\item The result follows by the point (\ref{lem: dis metriche coincidenti}) of the present Lemma, the fact that $\Det$ metrizes the weak convergence and the density in $\mathscr{M}(X)$ of the measures $\mu$ of the form $M\sum_{i=1}^n \delta_{x_i}$ with respect to weak convergence.
\vspace{0.3cm}

\item Let assume without loss of generality that $X=\{x_1,...,x_n\}$ and $X'=\{x'_1,...,x'_n\}$. We put $\delta=\sup_{i,j}|\mathsf{d}_{ij}-\mathsf{d}'_{ij}|$. We construct the following pseudo-metric coupling: on $X\times X$ we define $\hat{\di}=\di$, on $X'\times X'$ we put $\hat{\di}=\di'$, on $X\times X'$ we define
$$\hat{\di}(x_i,x'_j):=\inf_{k\in \{1,...,n\}}\di(x_i,x_k)+\di'(x'_k,x'_j)+\delta,$$
finally on $X'\times X$ we put
$$\hat{\di}(x'_i,x_j):=\inf_{k\in \{1,...,n\}}\di(x_j,x_k)+\di'(x'_k,x'_i)+\delta,$$
so that $\hat{\di}(x_i,x'_i)=\hat{\di}(x'_i,x_i)=\delta.$
\\We then define the measure coupling 
$$\boldsymbol{\gamma}=M\sum_{i=1}^n\delta_{(x_i,x'_i)}.$$
It is straightforward to see that $\hat{\di}$ and $\boldsymbol{\gamma}$ are actually couplings between $\di,\di'$ and $\mu,\mu'$, respectively. Then, using Proposition \ref{pr: formulazione D-LET} and recalling that $\ell$ is an increasing function we have that
$$\DET^{1/a}((X,\mathsf{d},\mu),(X',\mathsf{d}',\mu'))\leq \int_{X\times X'}\ell(\delta)\dd\boldsymbol{\gamma}=Mn\ell(\delta),$$
and the thesis follows.
\vspace{0.3cm}

\item We can take $\di$ itself as metric coupling. Then, by the point (\ref{lem: dis metriche coincidenti}) of the present Lemma, we have 
$$\DET^{1/a}((X,\mathsf{d},\mu),(X,\mathsf{d},M\mu))\le \Det^{1/a}(\mu,M\mu)=\et(\mu,M\mu).$$

By replacing the cost $\boldsymbol{\mathrm c}$ with the cost 
$$
\boldsymbol{\mathrm c}_{\infty}(x_1,x_2):=\begin{cases} 0 & \textrm{if } x_1=x_2\\
+\infty & \textrm{otherwise},
\end{cases}
$$
we obtain that 
$$\et(\mu,M\mu)\leq \et_{\infty}(\mu,M\mu),$$
where we have denoted by $\et_{\infty}$ the Entropy-Transport problem induced by the entropy function $F$ and the cost $\boldsymbol{\mathrm c}_{\infty}.$
Observe that every admissible entropy function satisfies
\begin{equation}\label{lem: bound F-variazione totale}
F(s)\leq C|s-1|, \quad \textrm{for every} \ \ 1/N<s<N,
\end{equation}
where 
$$C:=\max\left\{\frac{F(1/N)}{1/N-1},\frac{F(N)}{N-1}\right\}.$$

The conclusion now follows from an explicit computation of $\et_{\infty}$ together with the bound \eqref{lem: bound F-variazione totale}. Indeed, we have (see \cite[Example E.5]{LMS})
\begin{equation}
\et_{\infty}(\mu,M\mu)\le \min_{\theta\in [1,M]}\int_{X}C|\theta-1|+CM|\theta/M-1|\dd\mu=C\mu(X)|M-1|.
\end{equation}
\end{enumerate}
\end{proof}

The next Lemma shows the existence of the optimal couplings.  
\begin{lemma}\label{lem:esistenza ottimo}
Let $\Det$ be a regular Entropy-Transport distance induced by  $(a,F,\ell)$. Let $(X_1,\mathsf{d}_1,\mu_1)$ and $(X_2,\mathsf{d}_2,\mu_2)$ be two metric measure spaces. Then:
\begin{enumerate}[(i)]
\item There exist a measure $\boldsymbol{\gamma}\in\mathscr{M}(X_1\times X_2)$ and a pseudo-metric coupling $\hat{\mathsf{d}}$ between $\mathsf{d}_1$ and $\mathsf{d}_2$ such that
\begin{equation}
\DET^{1/a}((X_1,\mathsf{d}_1,\mu_1),(X_2,\mathsf{d}_2,\mu_2))=\sum_{i=1}^2D_F(\gamma_i||\mu_i)+\int_{X_1\times X_2}\ell\big(\hat{\mathsf{d}}(x,y)\big)\dd\boldsymbol{\gamma}.
\end{equation}
\item\label{esistenza metrico ottimo} There exist a complete and separable metric space $(\tilde{X},\tilde{\di})$ and isometric embeddings $\psi^1:\mathsf{supp}(\mu_1)\rightarrow \tilde{X}$, $\psi^2:\mathsf{supp}(\mu_2)\rightarrow \tilde{X}$ such that
\begin{equation}\label{eq: esistenza metrico ottimo}
\DET\big((X_1,\mathsf{d}_1,\mu_1),(X_2,\mathsf{d}_2,\mu_2)\big)= (\Det)_{\tilde{\di}}(\psi^1_{\sharp}\mu_1,\psi^2_{\sharp}\mu_2),
\end{equation}
where we have denoted by $(\Det)_{\tilde{\di}}$ the Entropy-Transport distance computed in the space $(\tilde{X},\tilde{\di})$.
\end{enumerate}
\end{lemma}
\begin{proof}
\begin{enumerate}[(i)]
\item 
$\mathbf{Step\, 1}$: tightness of the plans.\\
By Proposition \ref{pr: formulazione D-LET} there exist a sequence $\boldsymbol{\gamma}_n\in \mathscr{M}(X_1\times X_2)$ and $\hat{\mathsf{d}}_n$ pseudo-metric couplings of $\mathsf{d}_1,\mathsf{d}_2$ such that
\begin{equation}\label{dis: maggiorazione da def inf}
\sum_{i=1}^2D_F((\gamma_n)_i||\mu_i)+\int_{X_1\times X_2}\ell\big(\hat{\mathsf{d}}_n(x,y)\big)\dd\boldsymbol{\gamma}_n<\DET^{1/a}((X_1,\mathsf{d}_1,\mu_1),(X_2,\mathsf{d}_2,\mu_2))+\frac{1}{n}.
\end{equation}
Since the entropy functionals with respect to the fixed measures $\mu_1$ and $\mu_2$ are bounded, we can apply Theorems \ref{th: Prokhorov} and Lemma \ref{lem: compattezza sottolivelli entropia} in order to obtain the existence of subsequences (from now on we will not relabel them) such that $(\gamma_n)_i$ converges weakly to some $\gamma^i\in \mathscr{M}(X_i)$, $i=1,2$. Since $(\gamma_n)_i$ are marginals of the measure $\boldsymbol{\gamma}_n$, the tightness of $(\gamma_n)_i$ implies the tightness of $\boldsymbol{\gamma}_n$, so that the sequence $\boldsymbol{\gamma}_n\in \mathscr{M}(X_1\times X_2)$ is converging to some $\boldsymbol{\gamma}$. Moreover, by the continuity of the operator $\pi^i_{\sharp}$ with respect to the weak topology, the marginals of $\boldsymbol{\gamma}$ coincide with $\gamma^i$, $i=1,2.$ We notice that if $\boldsymbol{\gamma}$ is the null measure the proof is concluded by taking any pseudo-metric coupling $\hat{\mathsf{d}}$ between $\mathsf{d}_1$ and $\mathsf{d}_2$.

\medskip
$\mathbf{Step\, 2}$: pre-compactness of the pseudo-metric couplings.
\\Regarding the sequence $\hat{\mathsf{d}}_n$, by the triangle inequality we have that
$$|\hat{\mathsf{d}}_n(x_1,y_1)-\hat{\mathsf{d}}_n(x_2,y_2)|\leq |\mathsf{d}_1(x_1,x_2)+\mathsf{d}_2(y_1,y_2)|.$$
In particular, $\hat{\mathsf{d}}_n$ is uniformly $1$-Lipschitz with respect to the complete and separable metric $\mathsf{d}_1+\mathsf{d}_2$ on $X_1\times X_2$.
We claim it is also uniformly bounded in a point. To see this, take $(\bar{x},\bar{y})\in \mathsf{supp}(\boldsymbol{\gamma})$: since $\boldsymbol{\gamma}_n$ weakly converges to $\boldsymbol{\gamma}$ for every $r,\epsilon>0$ and for all $n$ sufficiently large we have
$$\boldsymbol{\gamma}_n\left(B_r(\bar{x})\times B_r(\bar{y})\right)\ge \boldsymbol{\gamma}\left(B_r(\bar{x})\times B_r(\bar{y})\right)-\epsilon.$$ 
Fix $r>0$ and  suppose by contradiction that there exists a subsequence (not relabeled) such that $2r\le \hat{\mathsf{d}}_n(\bar{x},\bar{y})\rightarrow +\infty$. For $\epsilon=\epsilon(r)$ small enough, from \eqref{dis: maggiorazione da def inf}, the fact that $(\bar{x},\bar{y})\in \mathsf{supp}(\boldsymbol{\gamma})$ and $\ell$ is increasing we infer the existence of some positive constants $C,c$ such that for all $n$ sufficiently large
\begin{multline*}
C>\int_{X_1\times X_2}\ell\left(\hat{\mathsf{d}}_n(x,y)\right)\dd\boldsymbol{\gamma}_n(x,y)\geq \int_{B_r(\bar{x})\times B_r(\bar{y})}\ell\left(\hat{\mathsf{d}}_n(\bar{x},\bar{y})-2r\right)\dd\boldsymbol{\gamma}_n(x,y)\\
\geq\ell\left(\hat{\mathsf{d}}_n(\bar{x},\bar{y})-2r\right)[\boldsymbol{\gamma}(B_r(\bar{x})\times B_r(\bar{y}))-\epsilon]\geq c\ell\left(\hat{\mathsf{d}}_n(\bar{x},\bar{y})-2r\right).
\end{multline*}
Since $\ell$ has bounded sublevels,  this implies that there exists a constant $K$ such that $\hat{\mathsf{d}}_n(\bar{x},\bar{y})<K$ for every $n$ that leads to a contradiction. 
\\ We can thus apply Ascoli-Arzelà's theorem to infer the existence of a limit function $\mathsf{d}:X_1\times X_2\rightarrow [0,\infty)$ such that $\mathsf{d}_n$ converges (up to subsequence) pointwise to $\mathsf{d}$ and the convergence is uniform on compact sets. We can extend $\mathsf{d}$ to $(X_1 \sqcup X_2)\times (X_1 \sqcup X_2)$ in order to get a limit pseudo-metric coupling, that we denote in the same way. 

\medskip
$\mathbf{Step\, 3}$: passing to the limit.\\
Next, we pass to the limit in the following expression
$$\sum_{i=1}^2D_F((\gamma_n)_i||\mu_i)+\int_{X_1\times X_2}\ell\big(\hat{\mathsf{d}}_n(x,y)\big)\dd\boldsymbol{\gamma}_n.$$
By Lemma \ref{lem: proprieta funzionale entropia}, the entropy is jointly lower semicontinuous and thus
\begin{equation*}
\liminf_n D_F((\gamma_n)_i||\mu_i)\geq D_F(\gamma_i||\mu_i).
\end{equation*}
So, it is sufficient to prove that
\begin{equation}\label{dis: semicontinuita' doppia successione}
\liminf_n \int_{X_1\times X_2}\ell\big(\hat{\mathsf{d}}_n(x,y)\big)\dd\boldsymbol{\gamma}_n \geq \int_{X_1\times X_2}\ell\big(\hat{\mathsf{d}}(x,y)\big)\dd\boldsymbol{\gamma}.
\end{equation}
Using the equi-tightness of $\{\boldsymbol{\gamma}_k\}$ we can find a sequence of compact sets $K_{1,n}\subset X_1$ and $K_{2,n}\subset X_2$ such that 
$$\boldsymbol{\gamma}_k\big(X_1\times X_2 \setminus (K_{1,n}\times K_{2,n})\big)\le\frac{1}{n}$$
for every $k$. 
We define $\ell_m(r):=\min(\ell(r),m),$ so that the sequence of functions $(x,y)\mapsto \ell_m(\mathsf{d}_n(x,y))$ converges uniformly on compact subsets of $X_1\times X_2$, as $n\to \infty$. Possibly by taking a further subsequence via a diagonal argument, we can infer that $\|\ell_m(\mathsf{d})-\ell_m(\mathsf{d}_n)\|_{\infty;n}\rightarrow 0$ when $n\rightarrow \infty$, where we denote by $\|\cdot \|_{\infty;n}$ the supremum norm in the set $K_{1,n}\times K_{2,n}.$
Let $M$ be a positive constant such that $\gamma_n(X_1\times X_2)\leq M$ for every $n$. We can bound the integral on the left hand side of \eqref{dis: semicontinuita' doppia successione} in the following way:
\begin{multline*}
\int_{X_1\times X_2} \ell(\hat{\mathsf{d}}_n) \dd\boldsymbol{\gamma}_n \ge
\int_{X_1\times X_2} \ell_m(\hat{\mathsf{d}}_n) \dd\boldsymbol{\gamma}_n \ge 
\int_{K^1_n\times K^2_n} \ell_m(\hat{\mathsf{d}}_n) \dd\boldsymbol{\gamma}_n \\
\ge 
 \int_{K^1_n\times K^2_n} \ell_m(\hat{\mathsf{d}}) \dd\boldsymbol{\gamma}_n - M \|\ell_m(\hat{\mathsf{d}})-\ell_m(\hat{\mathsf{d}}_n)\|_{\infty;n}   	\\		
\ge 
 \int_{X_1\times X_2} \ell_m(\hat{\mathsf{d}}) \dd\boldsymbol{\gamma}_n - M \|\ell_m(\hat{\mathsf{d}})-\ell_m(\hat{\mathsf{d}}_n)\|_{\infty;n} - m /n.
\end{multline*}
Now we can pass to the limit with respect to $n$ using the weak convergence of $\{\boldsymbol{\gamma}_n\}$, and we obtain
$$\liminf_{n} \int_{X_1\times X_2} \ell(\hat{\mathsf{d}}_n) \dd\boldsymbol{\gamma}_n\geq \int_{X_1\times X_2} \ell_m(\hat{\mathsf{d}}) \dd\boldsymbol{\gamma}$$
and then we conclude using the Beppo Levi's monotone convergence theorem with respect to $m$.

\item Without loss of generality we assume $\mathsf{supp}(\mu_i)=X_i$. By the previous point we know the existence of an optimal measure $\boldsymbol{\gamma}\in\mathscr{M}(X_1\times X_2)$ and an optimal pseudo-metric coupling $\hat{\mathsf{d}}$ between $\mathsf{d}_1$ and $\mathsf{d}_2$. We consider the complete and separable metric space $(\tilde{X},\tilde{\di})$ constructed as in Lemma \ref{lem: quotient disjoint union is metric}.
Denoting by $p: X_1\sqcup X_2\rightarrow \tilde{X}$ the projection to the quotient and using the identification
$$X_1\sqcup X_2=X_1\times\{0\}\cup X_2\times \{1\},$$
we notice that $X_1\times X_2 \hookrightarrow \tilde{X} \times \tilde{X}$ via the injective Borel map 
$$\boldsymbol{\psi}(x_1,x_2)=(\psi^1(x_1),\psi^2(x_2)):=(p(x_1,0),p(x_2,1)).$$
Moreover, we also have that $\psi^i$ is an isometry of $(X_i,\di_i)$ onto its image in $(\tilde{X},\tilde{\di})$, $i=1,2$. Thus, denoting by $\gamma_i$ the marginals of $\boldsymbol{\gamma}$, we can consider the measures $\boldsymbol{\psi}_{\sharp}\boldsymbol{\gamma}$ whose projections are $(\psi^1)_{\sharp}\gamma_1$ and $(\psi^2)_{\sharp}\gamma_2$. Using Lemma \ref{lem: invariance divergence under injection} we know that
\begin{equation}\label{eq: uguaglianza divergenze ottimali}
D_F(\gamma_i||\mu_i)=D_F((\psi^i)_{\sharp}\gamma_i\,||\,(\psi^i)_{\sharp}\mu_i), \qquad i=1,2.
\end{equation}
By recalling the definition of $\tilde{\di}$, we also have
\begin{equation}\label{eq: uguaglianza costi ottimali}
\int_{X_1\times X_2}\ell\big(\hat{\mathsf{d}}(x,y)\big)\dd\boldsymbol{\gamma}=\int_{\tilde{X}\times \tilde{X}}\ell\big(\tilde{\di}(x,y)\big)\dd(\boldsymbol{\psi}_{\sharp}\boldsymbol{\gamma}).
\end{equation}
Thus, as a consequence of \eqref{eq: uguaglianza divergenze ottimali},  \eqref{eq: uguaglianza costi ottimali} and the optimality of $\boldsymbol{\gamma}$ and $\hat{\di}$, the equality \eqref{eq: esistenza metrico ottimo} holds on $(\tilde{X},\tilde{\di})$ (with optimal measure $\boldsymbol{\psi}_{\sharp}\boldsymbol{\gamma}$). 
\end{enumerate}
\end{proof}
\begin{remark}
It is clear that the optimal coupling $\hat{\mathsf{d}}$ whose existence is proven in the previous Lemma is in general only a pseudo-metric and not a metric on $X_1\sqcup X_2$. To see this, it is sufficient to consider two isomorphic metric measure spaces $(X_1,\mathsf{d}_1,\mu_1)$, $(X_2,\mathsf{d}_2,\mu_2)$. If we denote by $\psi:X_1\rightarrow X_2$ the isometry between $(X_1,\mathsf{d}_1)$ and $(X_2,\mathsf{d}_2)$, the optimal coupling $\hat{\mathsf{d}}$ satisfies $\hat{\di}(x_1,\psi(x_1))=0$ for $\mu_1$-a.e $x_1$.
\end{remark}

The next theorem is the main result of the paper.
\begin{theorem}\label{th main}
Let $\Det$ be a regular Entropy-Transport distance induced by $(a,F,\ell)$. Then $(\boldsymbol{\mathrm{X}},\DET)$ is a complete and separable metric space. It is also a length (resp. geodesic) space if $\Det$ is a length (resp. geodesic) metric. 
\end{theorem}

\begin{proof}
$\mathbf{Step\, 1}$: $\DET$ defines a metric.\\
It is clear that $\DET$ is symmetric, finite valued, nonnegative and 
\begin{equation*}
\DET\big((X_1,\mathsf{d}_1,\mu_1),(X_2,\mathsf{d}_2,\mu_2)\big)=0 \ \  \textrm{if} \  \ (X_1,\mathsf{d}_1,\mu_1)=(X_2,\mathsf{d}_2,\mu_2).
\end{equation*}

We claim that $\DET\big((X_1,\mathsf{d}_1,\mu_1),(X_2,\mathsf{d}_2,\mu_2)\big)=0$ implies that the metric measure spaces $(X_1,\mathsf{d}_1,\mu_1)$ and $(X_2,\mathsf{d}_2,\mu_2)$ are isomorphic. By Lemma \ref{lem:esistenza ottimo} there exist a measure $\boldsymbol{\gamma}\in \mathscr{M}(X_1\times X_2)$ and a pseudo-metric coupling $\hat{\mathsf{d}}$ such that
$$0=\sum_{i=1}^2D_F\big(\gamma_i||\mu_i\big)+\int_{X_1\times X_2}\ell\big(\hat{\mathsf{d}}(x,y)\big)\dd\boldsymbol{\gamma}.$$
All the terms are nonnegative, so that $D_F\big(\gamma_i||\mu_i\big)=0$ and thus $\gamma_i=\mu_i$, $i=1,2$. Moreover, since $\ell(d)=0$ if and only if $d=0$, it follows that $\hat{\mathsf{d}}(x,y)=0$ for $\boldsymbol{\gamma}$-a.e $(x,y)$. We also have
\begin{equation}\label{eq:dxy=0gamma}
\hat{\mathsf{d}}(x,y)=0 \quad \text{for all } (x,y)\in  \mathsf{supp}(\boldsymbol{\gamma}).
\end{equation}
To see this, let $(\bar{x},\bar{y})\in  \mathsf{supp}(\boldsymbol{\gamma})$ so that for every $r>0$ we have $\boldsymbol{\gamma}(B_r(\bar{x},\bar{y}))>0$ where 
$$B_r((\bar{x},\bar{y}))=\{(x,y)\in X_1\times X_2\,|\, \mathsf{d}_1(\bar{x},x)+\mathsf{d}_2(\bar{y},y)<r\}.$$ 
We consider a sequence of balls of radius $r_n:=1/n, n\in \mathbb{N},$ and use the fact that $\hat{\mathsf{d}}(x,y)=0$ for $\boldsymbol{\gamma}$-a.e $(x,y)$ to infer the existence of a sequence of points $(x_n, y_n)\in B_{r_n}((\bar{x},\bar{y}))$ such that $\hat{\mathsf{d}}(x_n,y_n)=0$. Thus 
$$\hat{\mathsf{d}}(\bar{x},\bar{y})\le \mathsf{d}_1(\bar{x},x_n)+\hat{\mathsf{d}}(x_n,y_n)+\mathsf{d}_2(\bar{y},y_n)<1/n.$$
Sending $n\to +\infty$ and using the arbitrariness of $(\bar{x},\bar{y})$, the claim \eqref{eq:dxy=0gamma} follows.
\\Since $\mathsf{d}_1$ and $\mathsf{d}_2$ are metrics, we infer that for every $x_{1}\in  \mathsf{supp}(\mu_{1})$ there exists a unique  $x_{2}\in  \mathsf{supp}(\mu_{2})$ such that $(x_{1},x_{2})\in  \mathsf{supp}(\boldsymbol{\gamma})$. Indeed, for any $x_2,\tilde{x}_{2}\in  \mathsf{supp}(\mu_{2})$ such that $(x_{1},x_{2}), (x_{1},\tilde{x}_{2})\in  \mathsf{supp}(\boldsymbol{\gamma})$ we have 
$$\mathsf{d}_2(x_2,\tilde{x}_{2})=\hat{\mathsf{d}}(x_2,\tilde{x}_{2})\le \hat{\mathsf{d}}(x_{2},x_1)+\hat{\mathsf{d}}(\tilde{x}_{2}, x_1)=0$$ and thus $x_2=\tilde{x}_2$.
Switching the role of $X_{1}$ and $X_{2}$ in the argument above, we obtain the existence of  a  bijection $\psi:  \mathsf{supp}(\mu_{1}) \rightarrow  \mathsf{supp}(\mu_{2})$ such that $\boldsymbol{\gamma}=(\mathrm{Id},\psi)_{\sharp}\mu_1$ and (in virtue of \eqref{eq:dxy=0gamma})
\begin{equation}\label{eq:dxpsix=0}
\hat{\mathsf{d}}(x,\psi(x))=0 \quad \text{for all } x\in  \mathsf{supp}(\mu_{1}).
\end{equation}
Let $x,y\in \mathsf{supp}(\mu_{1}) $, from \eqref{eq:dxpsix=0} and the triangle inequality it follows
\begin{equation*}
\begin{aligned}
\mathsf{d_1}(x,y)=\hat{\mathsf{d}}(x,y)\le \hat{\mathsf{d}}(x,\psi(x))+\hat{\mathsf{d}}(\psi(x),\psi(y))+\hat{\mathsf{d}}(y,\psi(y))=\mathsf{d_2}(\psi(x),\psi(y)),\\
\mathsf{d_2}(\psi(x),\psi(y))=\hat{\mathsf{d}}(\psi(x),\psi(y))\le \hat{\mathsf{d}}(x,\psi(x))+\hat{\mathsf{d}}(x,y)+\hat{\mathsf{d}}(y,\psi(y))=\mathsf{d_1}(x,y),
\end{aligned}
\end{equation*}
which implies that $\psi:  \mathsf{supp}(\mu_{1}) \rightarrow  \mathsf{supp}(\mu_{2})$ is an isometry. 
\\Hence $(X_1,\mathsf{d}_1,\mu_1)$ and $(X_2,\mathsf{d}_2,\mu_2)$ are isomorphic, as claimed. 
\medskip

Regarding the triangle inequality, let $(X_i,\mathsf{d}_i,\mu_i)$, $i=1,2,3$, be three metric measure spaces. From the definition of $\DET$ and Proposition \ref{pr: formulazione D-LET}, for every $\epsilon>0$ we find a pseudo-metric coupling $\mathsf{d}_{12}$ between $\mathsf{d}_{1}$ and $\mathsf{d}_{2}$, and a pseudo-metric coupling $\mathsf{d}_{23}$  between $\mathsf{d}_{2}$ and $\mathsf{d}_{3}$ such that
\begin{align*}
&\DET\big((X_1,\mathsf{d}_1,\mu_1),(X_2,\mathsf{d}_2,\mu_2)\big)\geq (\Det)_{\di_{12}}(\mu_1,\mu_2)-\epsilon, \\
&\DET\big((X_2,\mathsf{d}_2,\mu_2),(X_3,\mathsf{d}_3,\mu_3)\big)\geq (\Det)_{\di_{23}}(\mu_2,\mu_3)-\epsilon,
\end{align*}
where we have denoted by $(\Det)_{\mathsf{d}}$ the Entropy-Transport distance induced by the pseudo-metric $\mathsf{d}$.
Set $X:=X_1\sqcup X_2\sqcup X_3$ and define a pseudo-metric $\mathsf{d}$ on $X$ in the following way
\begin{equation*}
\mathsf{d}(x,y):=\begin{cases}
\mathsf{d}_{12}(x,y) \ \ \ \ \ &\textrm{if} \ x,y\in X_1\sqcup X_2 \\
\mathsf{d}_{23}(x,y) \ \ \ \ \ &\textrm{if} \ x,y\in X_2\sqcup X_3 \\
\inf_{z\in X_2}[\mathsf{d}_{12}(x,z)+\mathsf{d}_{23}(z,y)] \ \ &\textrm{if} \ x\in X_1 \ \textrm{and} \ y\in X_3\\
\inf_{z\in X_2}[\mathsf{d}_{23}(x,z)+\mathsf{d}_{12}(z,y)] \ \ &\textrm{if}\  x\in X_3 \ \textrm{and} \ y\in X_1.
\end{cases}
\end{equation*}
We notice that $\mathsf{d}$ coincides with $\mathsf{d}_i$ when restricted to $X_i$. By applying Proposition \ref{pr: formulazione D-LET}, the point (\ref{lem: dis metriche coincidenti}) of Lemma \ref{lem: proprieta' di base} and the triangle inequality of $(\Det)_{\mathsf{d}}$ we obtain
\begin{multline*}
\DET\big((X_1,\mathsf{d}_1,\mu_1),(X_3,\mathsf{d}_3,\mu_3)\big)\leq (\Det)_{\mathsf{d}}(\mu_1,\mu_3)\leq (\Det)_{\mathsf{d}}(\mu_1,\mu_2)+(\Det)_{\mathsf{d}}(\mu_2,\mu_3)\\
=(\Det)_{\mathsf{d}_{12}}(\mu_1,\mu_2)+(\Det)_{\mathsf{d}_{23}}(\mu_2,\mu_3)\\
\leq \DET\big((X_1,\mathsf{d}_1,\mu_1),(X_2,\mathsf{d}_2,\mu_2)\big)+\DET\big((X_2,\mathsf{d}_2,\mu_2),(X_3,\mathsf{d}_3,\mu_3)\big)+2\epsilon.
\end{multline*}
The conclusion follows since $\epsilon>0$ is arbitrary.

\medskip
$\mathbf{Step\, 2}$: Completeness of $\DET$.\\
In order to prove completeness, let $\{(X_n,\mathsf{d}_n,\mu_n)\}_{n\in \mathbb{N}}$ be a Cauchy sequence in the space $(\boldsymbol{\mathrm{X}},\DET)$. In order to have convergence of the full sequence, it is enough to prove that there exists a converging subsequence. 
Let us consider a subsequence such that
\begin{equation*}
\DET^{1/a}\big((X_{n_k},\mathsf{d}_{n_k},\mu_{n_k}),(X_{n_{k+1}},\mathsf{d}_{n_{k+1}},\mu_{n_{k+1}})\big)<2^{-(k+1)}.
\end{equation*}
By definition of $\DET$ and Proposition \ref{pr: formulazione D-LET}, we can find a measure $\boldsymbol{\gamma}_{k+1}\in\mathscr{M}(X_{n_k}\times X_{n_{k+1}})$ and a complete and separable metric coupling $\hat{\mathsf{d}}_{k+1}$ between $\mathsf{d}_{X_{n_k}}$ and $\mathsf{d}_{X_{n_{k+1}}}$ such that 
\begin{equation}\label{eq: bound th completezza}
\int_{X_{n_k}} F(\sigma_{n_k})\dd\mu_{n_k}+\int_{X_{n_{k+1}}} F(\sigma_{n_{k+1}})\dd\mu_{n_{k+1}}+\int_{X_{n_k}\times X_{n_{k+1}}}\ell\big(\hat{\mathsf{d}}_{k+1}\big)\dd\boldsymbol{\gamma}_{k+1}<2^{-k},
\end{equation}
where $\sigma_{n_k}$ (resp. $\sigma_{n_{k+1}}$) is the Radon-Nykodim derivative of the first (resp. second) marginal of $\gamma_{k+1}$ with respect to $\mu_{n_k}$ (resp. $\mu_{n_{k+1}}$).

Now we want to define a sequence $\big\{(X'_k,\mathsf{d}'_k)\big\}_{k=1}^{\infty}$ of metric spaces such that $X_{n_k}\subset X'_k$ and $X'_k\subset X'_{k+1}$. We proceed in the following way: we set
\begin{align*}
& \big(X'_1,\mathsf{d}'_1\big):=\big(X_{n_1},\mathsf{d}_{X_{n_1}}\big), \\
& X'_{k+1}:=X'_k\sqcup X_{n_{k+1}} \big/ \sim,
\end{align*}
where $x\sim y$ if $\mathsf{d}'_{k+1}(x,y)=0$ and the latter is defined as 
\begin{equation*}
\mathsf{d}'_{k+1}(x,y):=
\begin{cases}
\mathsf{d}'_{k}(x,y) \ \ &\textrm{if} \ \ x,y\in X'_k\\ 
\hat{\mathsf{d}}_{k+1}(x,y) \ \ &\textrm{if} \ \ x,y\in X_{n_k}\sqcup X_{n_{k+1}} \\ 
\inf_{z\in X_{n_k}}\mathsf{d}'_{k}(x,z)+\hat{\mathsf{d}}_{k+1}(z,y) \ \ &\textrm{if} \ \ x\in X'_k,\,  y\in X_{n_{k+1}}\\
\inf_{z\in X_{n_k}}\mathsf{d}'_{k}(y,z)+\hat{\mathsf{d}}_{k+1}(z,x) \ \ &\textrm{if} \ \ y\in X'_k,\,  x\in  X_{n_{k+1}}.
\end{cases} 
\end{equation*} 
From the definition of $\mathsf{d}'_k$, it is clear that we can endow the space $X':=\bigcup_{k=1}^{\infty} X'_k$ with a limit metric $\mathsf{d}'$. Now we consider the completion $(X,\mathsf{d})$ of $(X',\mathsf{d}')$ and we notice that $(X_{n_k},\mathsf{d}_{X_{n_k}})$ is isometrically embedded in this space for every $k$. Using the embedding, we can also define a measure $\bar{\mu}_{n_k}$ as the push-forward of the measure $\mu_{n_k}.$ Combining  the construction above  with \eqref{eq: bound th completezza} gives
\begin{align}
&(\Det)^{1/a}_\mathsf{d}(\bar{\mu}_{n_k},\bar{\mu}_{n_{k+1}}) \nonumber \\ 
&\qquad \leq \int_{X_{n_k}} F(\sigma_{n_k})\dd\mu_{n_k}+\int_{X_{n_{k+1}}} F(\sigma_{n_{k+1}})\dd\mu_{n_{k+1}}+\int_{X_{n_k}\times X_{n_{k+1}}}\ell\left(\hat{\mathsf{d}}_{k+1}\right)\dd\boldsymbol{\gamma}_{k+1}<2^{-k}, \label{eq:Detmubar}
\end{align}
where $(\Det)_\mathsf{d}$ is the regular Entropy-Transport distance computed in the space $(X,\mathsf{d}).$ In particular, \eqref{eq:Detmubar} implies that  $(\bar{\mu}_{n_{k}})_{k\in \mathbb{N}}$ is a Cauchy sequence in $ (\mathscr M(X),(\Det)_\mathsf{d})$. Since $(\Det)_\mathsf{d}$ is complete, there exists $\mu\in \mathscr{M}(X)$ such that $(\Det)^{1/a}_\mathsf{d}(\bar{\mu}_{n_{k}}, \mu)\to 0$. 
 
Using again that $(X_{n_k},\mathsf{d}_{X_{n_k}})$ is isometrically embedded in $(X,\mathsf{d})$ and the point \eqref{lem: dis metriche coincidenti} of Lemma \ref{lem: proprieta' di base}, we can conlude that
\begin{equation}
\DET\big((X_{n_k},\mathsf{d}_{n_k},\mu_{n_k}),(X,\mathsf{d},\mu)\big)\leq (\Det)_\mathsf{d}(\bar{\mu}_{n_k},\mu)\rightarrow 0.
\end{equation}

\medskip

$\mathbf{Step\, 3}$: Separability of $\DET$.

Thanks to \eqref{lem: densita' delta} of Lemma \ref{lem: proprieta' di base} it is enough to show that the set $\boldsymbol{\mathrm{X}}_*$, defined in \eqref{eq: densita' delta}, is separable. To this aim, we notice that $\boldsymbol{\mathrm{X}}_*$ can be written as $\bigsqcup_{n\in \mathbb{N}} \tilde{\mathcal{K}}_n$ where
$$\tilde{\mathcal{K}}_n:=\{(X,\mathsf{d},\mu)\in \boldsymbol{\mathrm{X}}_*: \mathsf{supp}(\mu) \ \textrm{has} \ n \ \textrm{points}\}.$$ 
Since the set of all $(D,M)=(D_{ij},M)\in\mathbb{R}_+^{n\times n}\times \mathbb{R}_{+}$ such that 
\begin{equation}
D_{ij}=D_{ji} \, , \ \ D_{ij}=0 \iff i=j\, , \ \ D_{ij}\leq D_{ik}+D_{kj}\, 
\end{equation}
is separable (as a subset of the Euclidean space), using \eqref{lem: dis tra delta} of Lemma \ref{lem: proprieta' di base} we get that 
$$\tilde{\mathcal{K}}_{n,M}:=\{(X,\mathsf{d},\mu)\in \boldsymbol{\mathrm{X}}_*: \mathsf{supp}(\mu) \ \textrm{has} \ n \ \textrm{points and } \mu(X)=nM\}
$$
is separable for every fixed $n\in {\mathbb N}, M>0$. The separability of $\tilde{\mathcal{K}}_n$ follows by the separability of $\tilde{\mathcal{K}}_{n,M}$ combined with \eqref{lem: dis riscalamento misura} of Lemma \ref{lem: proprieta' di base}.

\medskip

$\mathbf{Step\, 4}$: Length/geodesic property of $\DET$.
\\Let us start by proving the length property. 
Let  $(X_1,\mathsf{d}_1,\mu_1),\, (X_2,\mathsf{d}_2,\mu_2) \in \boldsymbol{\mathrm{X}}$. By definition of $\DET$, for every $\varepsilon>0$ we can find a complete and separable metric space $(X, \di)$ and isometric embeddings $\psi^i: \mathsf{supp}(\mu_i) \to X$,  $i=1,2$, such that
\begin{equation}\label{eq:DETeps}
\DET((X_{1}, \di_{1},\mu_{1}), (X_{2}, \di_{2},\mu_{2}) )\geq  (\Det)_\mathsf{d}( \mu_{1}, \mu_{2})-\varepsilon,
\end{equation}
where, as before, we identify $\mathsf{supp}(\mu_i)$ with its isometric image $\psi^i(\mathsf{supp} (\mu_i))$, and correspondingly  $\mu_i$ with $\psi^i_\sharp \mu_{i}$, $i=1,2$, in order to keep notation short.
\\Recall that, by slightly modifying the classical Kuratowski embedding, one can show that every complete and separable metric space can be isometrically embedded in a complete, separable and \emph{geodesic} metric space (see for instance \cite[Exercise 1c. Ch. 3$\frac{1}{2}.1$]{Grom} or \cite[Proposition 1.2.12]{GiPa}). Thus, recalling also Lemma \ref{lem: invariance divergence under injection}, without loss of generality we can assume that the complete and separable metric space $(X,\di)$ above is also \emph{geodesic}.
\\By assumption $(\Det)_\mathsf{d}$ is a length distance on $\mathscr{M}(X)$ since $(X,\di)$ is a length space,  so that we can find a curve $(\mu_{t})_{t\in [1,2]}\subset (\mathscr{M}(X), (\Det)_\mathsf{d})$  from $\mu_1$ to $\mu_2$ satisfying   
\begin{equation}\label{eq:GeodDet}
\mathsf{Length}_{ (\Det)_\mathsf{d}} ((\mu_{t})_{t\in [1,2]})\le (\Det)_\mathsf{d}(\mu_{1}, \mu_{2})+\varepsilon.
\end{equation}
Now, it is easy to check that the $\DET$-length of the curve of m.m.s. $((X, \di, \mu_{t}))_{t\in [1,2]}\subset \boldsymbol{\mathrm{X}}$ satisfies
\begin{equation}\label{LengthETDet}
{\mathsf {Length}}_{ \DET} ( ((X, \di, \mu_{t}))_{t\in [1,2]} )  \leq   \mathsf{Length}_{ (\Det)_\mathsf{d}}( (\mu_{t})_{t\in [1,2]}).
\end{equation}
Indeed the length of a curve is by definition the supremum of the sums of mutual distances over finite partitions \eqref{eq:defLength}, and for every partition $(t_{i})$ of $[1,2]$ it holds
$$
\sum_{i} \DET((X, \di,\mu_{t_{i+1}}), (X, \di,\mu_{t_{i}}) ) \leq  \sum_{i} (\Det)_\mathsf{d}(\mu_{t_{i+1}}, \mu_{t_{i}} ) \leq \mathsf{Length}_{ (\Det)_\mathsf{d}}( (\mu_{t})_{t\in [1,2]}).
$$ 
The combination of \eqref{eq:DETeps}, \eqref{eq:GeodDet} and \eqref{LengthETDet} gives
\begin{align*}
\mathsf{Length}_{ \DET} ( ((X, \di, \mu_{t}))_{t\in [1,2]} ) & \leq   \mathsf{Length}_{ (\Det)_\mathsf{d}}( (\mu_{t})_{t\in [1,2]}) \le (\Det)_\mathsf{d}(\mu_{1},\mu_{2}) +\varepsilon \nonumber\\
&\leq  \DET((X_{1}, \di_{1},\mu_{1}), (X_{2}, \di_{2},\mu_{2})) +2\varepsilon,
\end{align*}
as desired.

To prove the geodesic property in the case $\Det$ is a geodesic distance, we notice that we can follow verbatim the argument given above with $\varepsilon=0$. Here one has to notice that the existence of an \emph{optimal} complete and separable metric space on which \eqref{eq:DETeps} holds with $\varepsilon=0$ follows from (\ref{esistenza metrico ottimo}) of Lemma \ref{lem:esistenza ottimo}.
\end{proof}

\begin{remark}
It is proved in \cite[Proposition 8.3]{LMS} that $(\mathscr{M}(X),\hk)$ is a geodesic space when the underlying space $(X,\di)$ is geodesic. In particular, the last claim of Theorem \ref{th main} can be applied to the Hellinger-Kantorovich distance. 
\\To the best of our knowledge, up to now this is the only known example of regular Entropy-Transport geodesic distance (with the trivial exception of weighted variants of $\hk$ \cite{LMS1}).
\end{remark}

\subsection{Topology}\label{subsec: topology}
Let us introduce a notion of convergence for sequences of (equivalence classes of) metric measure spaces (see \cite[Definition 3.9]{GMS} for the corresponding notion in the context of \emph{pointed} metric measure spaces).
\begin{definition}
We say that a sequence $(X_n,\mathsf{d}_n,\mu_n)_{n\in \mathbb{N}}$ weakly measured-Gromov converges to $(X_{\infty},\mathsf{d}_{\infty},\mu_{\infty})$ if there exist a complete and separable metric space $(X,\mathsf{d})$ and isometric embeddings $\iota_n:X_n\to X$, $n\in \bar{\mathbb{N}}$, such that $(\iota_n)_{\sharp}\mu_n \to (\iota_{\infty})_{\sharp}\mu_{\infty}$ weakly in $\mathscr{M}(X)$.
\end{definition}

In the next Theorem we see that this notion of convergence actually coincides with the convergence induced by any Sturm-Entropy-Transport distance.

\begin{theorem}\label{th: DET vs mG convergence}
Let $\Det$ be a regular Entropy-Transport distance induced by $(a,F,\ell)$. A sequence $(X_n,\mathsf{d}_n,\mu_n)_{n\in \mathbb{N}}$ weakly measured Gromov converges to $(X_{\infty},\mathsf{d}_{\infty},\mu_{\infty})$ if and only if 
\begin{equation}\label{eq: conv to 0 seq}
\DET\left((X_n,\mathsf{d}_n,\mu_n), (X_{\infty},\mathsf{d}_{\infty},\mu_{\infty})\right) \to 0 \quad \textrm{as} \ n \to \infty.
\end{equation}
\end{theorem}
\begin{proof}
Let us suppose the validity of \eqref{eq: conv to 0 seq}. By definition of $\DET$ we know that there exist a complete and separable metric space $(Y_n,\mathsf{d}_{Y_n})$ and isometric embeddings $\psi_n,\psi_n^{\infty}$ of $(X_n,\mathsf{d}_n)$, $(X_{\infty},\mathsf{d}_{\infty})$ respectively, in $Y_n$ such that 
\begin{equation}\label{eq: weak conv in Y_n}
\Det ((\psi_n)_{\sharp}\mu_n,(\psi_n^{\infty})_{\sharp}\mu_{\infty})<\frac{1}{n} \, ,
\end{equation}
where $\Det$ is computed in the space $Y_n$.
We now define $Y:=\sqcup_n X_n$, $n\in \bar{\mathbb{N}}$ endowed with the pseudo-metric $\mathsf{d}_Y$
\begin{equation*}
\mathsf{d}_Y(y,y'):=
\begin{cases}
\mathsf{d}_{n}(y,y') \ \ &\textrm{if} \ \ y,y'\in X_n, n\in \bar{\mathbb{N}}\\ 
\mathsf{d}_{Y_n}(\psi_n(y),\psi_n^{\infty}(y')) \ \ &\textrm{if} \ \ y\in X_{n}, \, y' \in X_{\infty} \\ 
\mathsf{d}_{Y_n}(\psi_n^{\infty}(y), \psi_n(y')) \ \ &\textrm{if} \  y \in X_{\infty}, \, y'\in X_{n}  \\
\inf_{x\in X_{\infty}}\mathsf{d}_{Y_n}(\psi_n(y), \psi_n^{\infty}(x))+\mathsf{d}_{Y_m}(\psi_m(y'), \psi_m^{\infty}(x)) \ \ &\textrm{if} \ \ y\in X_n,\,  y'\in  X_{m}.
\end{cases} 
\end{equation*}
We now consider the space $Y/\sim$ defined as the quotient of $Y$ with respect to the equivalence relation
\begin{equation}
y\sim y' \Leftrightarrow \mathsf{d}_Y(y,y')=0 \, ,
\end{equation}
and we then define the completion of this space, that we still denote by $(Y,\mathsf{d}_Y)$. It is easy to see that $Y$ is separable. By construction we notice that the set 
$$\psi_n(X_n)\cup \psi_n^{\infty}(X_{\infty})\subset Y_n$$ endowed with the distance $\mathsf{d}_{Y_n}$ is canonically isometrically embedded in $(Y,\mathsf{d}_Y)$, so that every space $X_n$, $n\in \bar{\mathbb{N}}$, is canonically isometrically embedded into $Y$ by a map $\psi_n'$. We claim now that $Y$ and $\psi_n'$ provide a realization of the weakly measured Gromov convergence. To see this, it is enough to notice that $(\psi_n')_{\sharp}\mu_n\to (\psi_{\infty}')_{\sharp}\mu_{\infty}$ weakly in $\mathscr{M}(Y)$ which is a consequence of the construction of $\psi_n'$, \eqref{eq: weak conv in Y_n} and the fact that $\Det$ induces the weak topology.\\
For the converse, let us suppose that $(X_n,\mathsf{d}_n,\mu_n)_{n\in \mathbb{N}}$ weakly measured Gromov converges to $(X_{\infty},\mathsf{d}_{\infty},\mu_{\infty})$. By definition we know that there exist a complete and separable metric space $(X,\mathsf{d})$ and isometric embeddings $\iota_n:X_n\to X$, $n\in \bar{\mathbb{N}}$, such that $(\iota_n)_{\sharp}\mu_n \to (\iota_{\infty})_{\sharp}\mu_{\infty}$ weakly in $\mathscr{M}(X)$. Since $\Det$ metrizes the weak convergence on $\mathscr{M}(X)$ we know that 
$$\Det((\iota_n)_{\sharp}\mu_n, (\iota_{\infty})_{\sharp}\mu_{\infty})\rightarrow 0 \qquad \textrm{as} \ n \to \infty,$$
and the result follows by the very definition of $\DET$, noticing that $(X,\mathsf{d})$ is a possible competitor.
\end{proof}

Let us denote by $\boldsymbol{\mathrm{X}}(K,N,L,v,V)$ the family of isomorphism classes of metric measure spaces $(X,\mathsf{d},\mu) \in \mathsf{CD}(K,N)$ such that 
$$\mathsf{diam}(X)\le L \quad \textrm{and} \quad 0<v\le \mu(X)\le V.$$ 
Let $\boldsymbol{\tilde{\mathrm{X}}}(K,N,L,v,V)$ be the family of isomorphism classes of spaces 
$$(X,\mathsf{d},\mu)\in\boldsymbol{\mathrm{X}}(K,N,L,v,V)$$
such that $\mu$ has full support.
\begin{theorem}\label{th: compact CD}
Fix $K\in \mathbb{R}$, $N\in (1,\infty)$, $L\in (0,\infty)$ and $0<v\le V<\infty$. Let $\Det$ be a regular Entropy-Transport distance. Then
\begin{itemize}
\item $\boldsymbol{\mathrm{X}}(K,N,L,v,V)$ is compact with respect to $\DET$. 
\item $\boldsymbol{\tilde{\mathrm{X}}}(K,N,L,v,V)$ is compact with respect to mGH. Moreover on such family the $\DET$-topology and the mGH-topology coincide.
\end{itemize}
\end{theorem}
\begin{proof}
By \cite[Corollary 3.22]{GMS} we have precompactness of $\boldsymbol{\mathrm{X}}(K,N,L,v,V)$ with respect to the weakly measured Gromov convergence and thus precompactness with respect also to the $\DET$-convergence by Theorem \ref{th: DET vs mG convergence}. From \cite[Theorem 3.1]{St2} (see also \cite[Theorem 4.9]{GMS}) we know that the condition $\mathsf{CD}(K,N)$ is stable with respect to the weakly measured Gromov convergence and thus the first statement follows.
For the second statement we observe that the spaces in $\boldsymbol{\tilde{\mathrm{X}}}(K,N,L,v,V)$ are uniformly doubling and thus the weakly measured Gromov convergence is equivalent to the mGH-convergence (see \cite[Theorem 3.30 and 3.33]{GMS}).
\end{proof}

\begin{corollary}
It follows:
\begin{enumerate}[(i)]
\item The stability of $\mathsf{CD}(K,N)$ with $N\in (1,\infty]$ under $\DET$-convergence.
\item The convergence of heat flows under $\DET$-convergence of $\mathsf{CD}(K,\infty)$ spaces.
\item The stability of $\mathsf{RCD}(K,N)$ with $N\in (1,\infty]$ under $\DET$-convergence.
\item The stability of the spectrum of the Laplacian under $\DET$-convergence of $\mathsf{CD}(K,\infty)$ spaces.   
\end{enumerate}
\end{corollary}
\begin{proof}
The proof is a direct consequence of Theorem \ref{th: DET vs mG convergence} and the results contained in \cite{GMS}, to which we refer for the precise statements. In particular, for $(i)$ we use \cite[Theorem I and pp. 29-30]{GMS}, $(ii)$ follows from \cite[Theorem 5.7]{GMS}, for $(iii)$ we take advantage of \cite[Theorem IV]{GMS}, and $(iv)$ is a consequence of \cite[Theorem V]{GMS}.
\end{proof}

\section{Limiting cases}
\subsection{Pure entropy distances}

In the setting of Entropy-Transport problems, we call \emph{pure entropy} problems the ones induced by the choices
$$F\in \Gamma(\mathbb{R}_+), \qquad \boldsymbol{\mathrm c}(x_1,x_2)=\begin{cases}0 \ &\textrm{if} \ x_1=x_2, \\ +\infty &\textrm{otherwise.} \end{cases}$$

In this situation one can prove (see \cite[Example E.5]{LMS}) that for any $ \mu_1,\mu_2\in \Mi(X)$ we have
\begin{equation}\label{eq: PE problem}
\et(\mu_1,\mu_2)=\inf_{\gamma\in \Mi(X)}D_F(\gamma||\mu_1)+D_F(\gamma||\mu_2)=\int_{X}H_0\left(\frac{\mathrm{d}\mu_1}{\mathrm{d}\lambda},\frac{\mathrm{d}\mu_2}{\mathrm{d}\lambda}\right)\dd\lambda,
\end{equation}
where $\lambda\in \Mi(X)$ is any dominating measure of $\mu_1$ and $\mu_2$ and $H_0$ is defined as the lower semicontinuous envelope of the function 
\begin{equation}
\tilde{H}_0(r,t):=\inf_{\theta>0} \hat{F}(\theta,r)+\hat{F}(\theta,t).
\end{equation}
In particular, the functional $\et$ corresponds in this situation to the Csisz\'ar's divergence induced by the function $s\mapsto H_0(1,s)\in \Gamma_0(\mathbb{R}_+)$ (see \cite[Lemma 3]{DepIeee}), justifying the name of pure entropy problem.

For some entropy functions $F$ one can prove that a power $a$ of the induced pure entropy cost $\et$ is a distance. For instance, when $a=1$ and $F(s)=|s-1|$ we obtain the celebrated total variation (denoted by $\mathsf{TV}$ in the sequel), a distance in the space of measures inducing a strong topology. 
Actually, thanks to the result proved in \cite[Lemma 8]{DepIeee} and the explicit bounds contained in \cite[Theorem 2.5]{LS}, we know that every pure entropy distance induces the same topology of the total variation.    

As shown in \cite[Propositions 2, 3]{DepIeee}, we obtain another class of pure entropy distances by choosing $a=1/2$ and the power-like entropy $F=U_p$, $p\ge 1$, defined in example \ref{def:power-like}. In this situation we have 
\begin{align}
&H_0(r,t)=\frac{1}{p}\Big[r+t-2^{\frac{p}{p-1}}(r^{1-p}+t^{1-p})^{\frac{1}{1-p}}\Big] \qquad \textrm{if} \ p>1,\\
&H_0(r,t)=(\sqrt{r}-\sqrt{t})^2, \qquad p=1,
\end{align}
and we recognize some well-known functionals like the $2$-Hellinger distance (case $p=1$) and the triangular discrimination (case $p=2$). We will denote these distances by $\mathsf{PL}_{p}$.

We start with a useful lemma, valid for any pure entropy problem.
\begin{lemma}\label{lem: explicit formulation pure entropy}
Fix $a\in(0,1]$ and let us consider the functions
$$F\in \Gamma(\mathbb{R}_+) \ \textrm{such that} \ F^{\prime}_{\infty}=+\infty, \qquad \boldsymbol{\mathrm c}(x_1,x_2)=\begin{cases}0 \ &\textrm{if} \ x_1=x_2, \\ +\infty &\textrm{otherwise}. \end{cases}$$
Let us denote by $\mathsf{PE}$ the power $a$ of the Entropy-Transport cost induced by $F$ and $\boldsymbol{\mathrm c}$.
\\For any $(X_1,\mathsf{d}_1,\mu_1),(X_2,\mathsf{d}_2,\mu_2)\in \boldsymbol{\mathrm{X}}$ let us define
\begin{equation}\label{eq: def sturm pure entropy}
\mathbf{PE}\left((X_1,\mathsf{d}_1,\mu_1),(X_2,\mathsf{d}_2,\mu_2)\right):=\inf \mathsf{PE}(\psi^1_{\sharp}\mu_1,\psi^2_{\sharp}\mu_2),
\end{equation}
where the infimum in the right hand side is taken over all complete and separable metric spaces $(\hat{X},\hat{\mathsf{d}})$ with isometric embeddings $\psi^1:\mathsf{supp}(\mu_1)\rightarrow \hat{X}$ and $\psi^2:\mathsf{supp}(\mu_2)\rightarrow \hat{X}$.
\\Then 
\begin{equation}\label{eq: explicit formulation PE}
\mathbf{PE}^{1/a}\left((X_1,\mathsf{d}_1,\mu_1),(X_2,\mathsf{d}_2,\mu_2)\right)=\inf_C\left\{D_F(\gamma_1||\mu_1)+D_F(\gamma_2||\mu_2)\right\}  ,
\end{equation}
where 
\begin{equation*}
C:=\left\{(\boldsymbol{\gamma},\hat{\di}): \boldsymbol{\gamma}\in \Mi(X_1\times X_2),\, \hat{\di} \ \textrm{pseudo-metric coupling for} \ \di_1,\di_2, \ \mathsf{supp}(\boldsymbol{\gamma})\subset \{\hat{\di}=0\} \right\}.
\end{equation*}
\end{lemma}
\begin{proof}
Setting
$$\ell(d):=\begin{cases} 0 &\textrm{if} \ d=0\\
+\infty &\textrm{otherwise},
\end{cases}$$
we can prove that the infimum of 
$$\left\{D_F(\gamma_1||\mu_1)+D_F(\gamma_2||\mu_2)+\int_{X_1\times X_2}\ell\big(\hat{\mathsf{d}}(x,y)\big)\dd\boldsymbol{\gamma}\right\}^a  $$
over the set
$$\tilde{C}:=\{(\boldsymbol{\gamma},\hat{\di}): \boldsymbol{\gamma}\in \Mi(X_1\times X_2),\, \hat{\di} \ \textrm{pseudo-metric coupling for} \ \di_1,\di_2\}$$
is less or equal to the infimum in the right hand side of \eqref{eq: def sturm pure entropy} by reasoning as in the first part of the proof of Proposition \ref{pr: formulazione D-LET}. The fact that the power $a$ of the right hand side of \eqref{eq: explicit formulation PE} is less or equal to the infimum as in \eqref{eq: def sturm pure entropy} follows by noticing that
\begin{equation}
\int_{X_1\times X_2}\ell\big(\hat{\mathsf{d}}(x,y)\big)\dd\boldsymbol{\gamma}=
\begin{cases}
+\infty \qquad &\textrm{if} \ \mathsf{supp}(\boldsymbol{\gamma})\not\subset \{\hat{\di}=0\}\\
0 &\textrm{otherwise}.
\end{cases}
\end{equation}

For the converse inequality,  we reason in a similar way as in the proof of the point (\ref{esistenza metrico ottimo}) of Lemma \ref{lem:esistenza ottimo}. For any pseudo-metric coupling $\hat{\di}$ of $\di_1$ and $\di_2$ let us consider the space $((X_1\sqcup X_2)/\sim,\hat{\di})$, where $x_1\sim x_2 \Longleftrightarrow \hat{\di}(x_1,x_2)=0.$
It is a complete and separable metric space as proved in Lemma \ref{lem: quotient disjoint union is metric}.
Denoting by $q: X_1\sqcup X_2\rightarrow (X_1\sqcup X_2)/\sim$ the projection to the quotient and using the identification
$$X_1\sqcup X_2=X_1\times\{0\}\cup X_2\times \{1\},$$
we notice that
$$X_1\times X_2 \hookrightarrow \big((X_1\sqcup X_2)/\sim\big) \times \big((X_1\sqcup X_2)/\sim\big)$$
via the injective map 
$$\boldsymbol{\psi}(x_1,x_2)=(\psi^1(x_1),\psi^2(x_2)):=(q(x_1,0),q(x_2,1)).$$
Moreover, we also have that $\psi^i$ is an isomorphism of $(X_i,\di_i)$ into its image in $((X_1\sqcup X_2)/\sim,\hat{\di})$, $i=1,2$. Thus, for any measure $\boldsymbol{\gamma}\in \Mi(X_1\times X_2)$ such that $\mathsf{supp}(\boldsymbol{\gamma})\subset \{\hat{\di}=0\}$, denoting by $\gamma_i$ the marginals of $\boldsymbol{\gamma}$, we can consider the measures $\boldsymbol{\psi}_{\sharp}\boldsymbol{\gamma}$ whose projections are $(\psi^1)_{\sharp}\gamma_1$ and $(\psi^2)_{\sharp}\gamma_2$. Using Lemma \ref{lem: invariance divergence under injection} we know that
$$D_F(\gamma_i||\mu_i)=D_F((\psi^i)_{\sharp}\gamma_i||(\psi^i)_{\sharp}\mu_i), \qquad i=1,2$$
and the proof is completed by noticing that $\mathsf{supp}(\boldsymbol{\psi}_{\sharp}\boldsymbol{\gamma})$ is contained in the diagonal of the metric space $((X_1\sqcup X_2)/\sim,\hat{\di})$.
\end{proof}

In the next theorem we prove that some pure entropy problems, specifically the ones generated by power-like entropies $U_p$, $p\ge 1$, can be recovered as a limiting case of regular Entropy-Transport problems.

\begin{theorem}\label{th: limit entropy distances}
Fix $p\ge 1$ and let us consider the sequence of cost functions $\boldsymbol{\mathrm c}_n=n\di$ and the entropy function $F:=U_p$. 
Let us denote by $\mathsf{D}_{p,n}$ the Entropy-Transport distance induced by $a=1/2$, $\boldsymbol{\mathrm c}_n=n\mathsf{d}$ and $F:=U_p$. 

Then, for every metric measure spaces $(X_1,\mathsf{d}_1,\mu_1)$, $(X_2,\mathsf{d}_2,\mu_2)\in \boldsymbol{\mathrm{X}}$ the limit
\begin{equation}\label{eq: limit entropy distances}
\mathbf{PL}_p\left((X_1,\mathsf{d}_1,\mu_1),(X_2,\mathsf{d}_2,\mu_2)\right):=\lim_{n \to \infty}\mathbf{D}_{p,n}\left((X_1,\mathsf{d}_1,\mu_1),(X_2,\mathsf{d}_2,\mu_2)\right) \quad \textrm{is well defined},
\end{equation}
where $\mathbf{D}_{p,n}$ denotes the function defined as in Definition \ref{def: D-ET} upon replacing $\Det$ by $\mathsf{D}_{p,n}$.

Moreover, $\mathsf{D}_{p,n}$ is a regular Entropy-Transport distance and $\mathbf{PL}_p$ defines a metric on $\boldsymbol{\mathrm{X}}$ such that
\begin{equation}\label{eq: conv to entropy distances}
\mathbf{PL}_p\left((X_1,\mathsf{d}_1,\mu_1),(X_2,\mathsf{d}_2,\mu_2)\right)=\inf \mathsf{PL}_{p}(\psi^1_{\sharp}\mu_1,\psi^2_{\sharp}\mu_2),
\end{equation}
where the infimum in the right hand side is taken over all complete and separable metric spaces $(\hat{X},\hat{\mathsf{d}})$ with isometric embeddings $\psi^1:\mathsf{supp}(\mu_1)\rightarrow \hat{X}$ and $\psi^2:\mathsf{supp}(\mu_2)\rightarrow \hat{X}$.
\end{theorem}
\begin{proof}
The first assertion follows by noticing that for any metric $\mathsf{d}$ we have
\begin{equation}\label{eq: monotonicity cost}
n\mathsf{d}(x_1,x_2)\uparrow \boldsymbol{\mathrm c}(x_1,x_2)=\begin{cases}0 \ &\textrm{if} \ x_1=x_2 \\ +\infty &\textrm{otherwise} \end{cases} \qquad \textrm{as} \ n \to \infty \ \textrm{for every} \ x_1,x_2\in X.
\end{equation}
The fact that $\mathsf{D}_{p,n}$ is a regular Entropy-Transport distance is a consequence of \cite[Theorem 6]{DepIeee}, noticing the obvious fact that $n\mathsf{d}$ is a complete and separable metric for any fixed $n$. 

In particular, since for every fixed $n$ we know that $\mathbf{D}_{p,n}$ is a metric on $\boldsymbol{\mathrm{X}}$ by Theorem \ref{th main}, we have that $\mathbf{PL}_p$ is nonnegative, symmetric, it satisfies the triangle inequality and 
$$\mathbf{PL}_p\left((X_1,\mathsf{d}_1,\mu_1),(X_2,\mathsf{d}_2,\mu_2)\right)=0 \quad \textrm{if} \quad (X_1,\mathsf{d}_1,\mu_1)=(X_2,\mathsf{d}_2,\mu_2).$$
We claim that 
$$\mathbf{PL}_p\left((X_1,\mathsf{d}_1,\mu_1),(X_2,\mathsf{d}_2,\mu_2)\right)=0$$
only if $(X_1,\mathsf{d}_1,\mu_1)=(X_2,\mathsf{d}_2,\mu_2)$ (as equivalence classes). Indeed, since $\mathbf{D}_{p,n}$ is nonnegative and nondecreasing, the fact that 
$$\lim_{n\to \infty}\mathbf{D}_{p,n}\left((X_1,\mathsf{d}_1,\mu_1),(X_2,\mathsf{d}_2,\mu_2)\right)=0$$
implies
$$\mathbf{D}_{p,n}\left((X_1,\mathsf{d}_1,\mu_1),(X_2,\mathsf{d}_2,\mu_2)\right)=0 \qquad \textrm{for every} \ n,$$
and the result follows because $\mathbf{D}_{p,n}$ is a distance on $\boldsymbol{\mathrm{X}}$.

At this level we do not know that $\mathbf{PL}_p$ is finite valued, which is a consequence of \eqref{eq: conv to entropy distances} together with the fact that $\mathsf{PL}_p$ is a (finite valued) distance on the space of measures as recalled above.  

In order to prove \eqref{eq: conv to entropy distances}, we first notice that the monotonicity \eqref{eq: monotonicity cost} easily implies that
$$\mathbf{PL}_p\left((X_1,\mathsf{d}_1,\mu_1),(X_2,\mathsf{d}_2,\mu_2)\right)\le \inf \mathsf{PL}_{p}(\psi^1_{\sharp}\mu_1,\psi^2_{\sharp}\mu_2),$$
giving the finiteness of  $\mathbf{PL}_p\left((X_1,\mathsf{d}_1,\mu_1),(X_2,\mathsf{d}_2,\mu_2)\right)$.

For the converse inequality, thanks to Lemma \ref{lem:esistenza ottimo} we know that for every $(X_1,\mathsf{d}_1,\mu_1)$, $(X_2,\mathsf{d}_2,\mu_2)$ and for every $n \in {\mathbb N}$ there exist a measure $\boldsymbol{\gamma}_n\in \Mi(X_1\times X_2)$ and a pseudo-metric coupling $\hat{\mathsf{d}}_n$ between $\mathsf{d}_1$ and $\mathsf{d}_2$ such that:
\begin{multline}\label{eq: explicit D^2_p,n small than C}
\infty> \mathbf{PL}^2_p\left((X_1,\mathsf{d}_1,\mu_1),(X_2,\mathsf{d}_2,\mu_2)\right)\ge \mathbf{D}^2_{p,n}((X_1,\mathsf{d}_1,\mu_1),(X_2,\mathsf{d}_2,\mu_2))\\
= \sum_{i=1}^2D_{U_p}(\gamma_{n,i}||\mu_i)+\int_{X_1\times X_2}n\hat{\mathsf{d}}_n(x,y)\,\dd\boldsymbol{\gamma_n} \qquad \textrm{for every} \ n\in \mathbb{N}.
\end{multline} 
By the superlinearity of the entropy functionals we can infer the existence (up to subsequence) of a weak limit ${\boldsymbol{\gamma}}\in \Mi(X_1\times X_2)$ of the sequence $\{{\boldsymbol{\gamma}_n}\}_{n\in \mathbb{N}}$. We also know that
\begin{equation}
\liminf_{n\to \infty} \sum_{i=1}^2D_{U_p}(\gamma_{n,i}||\mu_i) \geq \sum_{i=1}^2D_{U_p}(\gamma_i||\mu_i),
\end{equation}
where we have used the usual notation for the marginal measures.
If $\boldsymbol{\gamma}$ is the null measure the result follows trivially. Otherwise, since the integral 
$$\int_{X_1\times X_2}\hat{\mathsf{d}}_n(x,y)\,\dd\boldsymbol{\gamma_n}$$
is bounded from above we can argue as in the step $2$ of the proof of Lemma \ref{lem:esistenza ottimo} and we deduce the existence of a pseudo-metric coupling $\hat{\mathsf{d}}$ between $\mathsf{d}_1$ and $\mathsf{d}_2$ such that $\hat{\mathsf{d}}_n$ converges (up to subsequence) pointwise to $\hat{\mathsf{d}}$ and the convergence is uniform on compact sets. 
\\By recalling the explicit formulation of the right hand side of \eqref{eq: conv to entropy distances} given in Lemma \ref{lem: explicit formulation pure entropy}, the proof is completed if we show that $\mathsf{supp}(\boldsymbol{\gamma})\subset \{\hat{\di}=0\}$.
Let us suppose by contradiction the existence of a point $(\bar{x},\bar{y})\in \mathsf{supp}(\boldsymbol{\gamma})$ such that $\hat{\di}(\bar{x},\bar{y})=k>0.$
Fix $k/2>r>0$: for every $\epsilon>0$ sufficiently small we know that there exist $m\in \mathbb{N}$ such that for every $n>m$ we have
\begin{equation}\label{eq: two conditions positivity}
\begin{aligned}
&\boldsymbol{\gamma}_n\left(B_r(\bar{x})\times B_r(\bar{y})\right)\ge \boldsymbol{\gamma}\left(B_r(\bar{x})\times B_r(\bar{y})\right)-\epsilon>0,  \\ 
&\hat{\mathsf{d}}_n(\bar{x},\bar{y})-2r>\hat{\mathsf{d}}(\bar{x},\bar{y})-2r-\epsilon>0.
\end{aligned}
\end{equation} 
Starting from the bound in \eqref{eq: explicit D^2_p,n small than C} we have
\begin{align*}
\infty&> \mathbf{PL}^2_p\left((X_1,\mathsf{d}_1,\mu_1),(X_2,\mathsf{d}_2,\mu_2)\right)\geq n\int_{X_1\times X_2}\hat{\mathsf{d}}_n(x,y)\,\dd\boldsymbol{\gamma_n}\\
&\ge n\int_{B_r(\bar{x})\times B_r(\bar{y})}\hat{\mathsf{d}}_n(x,y)\,\dd\boldsymbol{\gamma_n} \ge n\left(\hat{\mathsf{d}}_n(\bar{x},\bar{y})-2r\right)\boldsymbol{\gamma}_n(B_r(\bar{x})\times B_r(\bar{y})) \\
& \geq n\left(\hat{\mathsf{d}}(\bar{x},\bar{y})-2r-\epsilon \right)[\boldsymbol{\gamma}(B_r(\bar{x})\times B_r(\bar{y}))-\epsilon]
\end{align*} 
that leads to a contradiction for $n$ sufficiently large thanks to \eqref{eq: two conditions positivity}.
\end{proof}

\begin{definition}
We say that a sequence of metric measure spaces $(X_n,\mathsf{d}_n,\mu_n)_{n\in \mathbb{N}}$ strongly measured-Gromov converges to the metric measure space $(X_{\infty},\mathsf{d}_{\infty},\mu_{\infty})$ if there exist a complete and separable metric space $(X,\mathsf{d})$ and isometric embeddings $\iota_n:X_n\to X$, $n\in \bar{\mathbb{N}}$, such that $(\iota_n)_{\sharp}\mu_n \to (\iota_{\infty})_{\sharp}\mu_{\infty}$ in $\mathscr{M}(X)$ with respect to the total variation topology.
\end{definition}

In the next Theorem we see that this notion of convergence coincides with the convergence induced by the distance $\mathbf{PL}_p$ for every $p\ge 1$.

\begin{theorem}\label{th: SPL vs mG convergence}
Let $p\ge 1$. A sequence $(X_n,\mathsf{d}_n,\mu_n)_{n\in \mathbb{N}}$ strongly measured-Gromov converges to $(X_{\infty},\mathsf{d}_{\infty},\mu_{\infty})$ if and only if 
\begin{equation}\label{eq: strong conv to 0 seq}
\mathbf{PL}_p\left((X_n,\mathsf{d}_n,\mu_n), (X_{\infty},\mathsf{d}_{\infty},\mu_{\infty})\right) \to 0 \quad \textrm{as} \ n \to \infty.
\end{equation}
\end{theorem}
\begin{proof}
The proof is analogous to the one of Theorem \ref{th: DET vs mG convergence}.
Let us suppose the validity of \eqref{eq: strong conv to 0 seq}. By definition of $\mathbf{PL}_p$ we know that there exist a complete and separable metric space $(Y_n,\mathsf{d}_{Y_n})$ and isometric embeddings $\psi_n,\psi_n^{\infty}$ of $(X_n,\mathsf{d}_n)$ and $(X_{\infty},\mathsf{d}_{\infty})$ in $Y_n$ such that 
\begin{equation}\label{eq: strong conv in Y_n}
\mathsf{PL}_p ((\psi_n)_{\sharp}\mu_n,(\psi_{\infty})_{\sharp}\mu_{\infty})<\frac{1}{n} \, ,
\end{equation}
where $\mathsf{PL}_p$ is computed in the space $Y_n$.
We now define $\displaystyle Y:=\sqcup_{n\in \bar{\mathbb{N}}} X_n$ endowed with the pseudo-metric $\mathsf{d}_Y$ 
\begin{equation*}
\mathsf{d}_Y(y,y'):=
\begin{cases}
\mathsf{d}_{n}(y,y') \ \ &\textrm{if} \ \ y,y'\in X_n, n\in \bar{\mathbb{N}}\\ 
\mathsf{d}_{Y_n}(\psi_n(y),\psi_n^{\infty}(y')) \ \ &\textrm{if} \ \ y\in X_{n}, \, y' \in X_{\infty} \\ 
\mathsf{d}_{Y_n}(\psi_n^{\infty}(y), \psi_n(y')) \ \ &\textrm{if} \  y \in X_{\infty}, \, y'\in X_{n}  \\
\inf_{x\in X_{\infty}}\mathsf{d}_{Y_n}(\psi_n(y), \psi_n^{\infty}(x))+\mathsf{d}_{Y_m}(\psi_m(y'), \psi_m^{\infty}(x)) \ \ &\textrm{if} \ \ y\in X_n,\,  y'\in  X_{m}.
\end{cases} 
\end{equation*}
We consider the space $Y/\sim$ defined as the quotient of $Y$ with respect to the equivalence relation
\begin{equation}
y\sim y' \Leftrightarrow \mathsf{d}_Y(y,y')=0 \, ,
\end{equation}
and we then define the completion of this space, that we still denote by $(Y,\mathsf{d}_Y)$. It is easy to see that $Y$ is separable. By construction we notice that the set 
$$\psi_n(X_n)\cup \psi_n^{\infty}(X_{\infty})\subset Y_n$$ endowed with the distance $\mathsf{d}_{Y_n}$ is canonically isometrically embedded in $(Y,\mathsf{d}_Y)$, so that every space $X_n$, $n\in \bar{\mathbb{N}}$, is canonically isometrically embedded into $Y$ by a map $\psi_n'$. We claim now that $Y$ and $\psi_n'$ provide a realization of the strong measured-Gromov convergence. To see this, it is enough to notice that $(\psi_n')_{\sharp}\mu_n\to (\psi_{\infty}')_{\sharp}\mu_{\infty}$ in $\mathscr{M}(Y)$ with respect to the topology induced by the total variation, which is a consequence of the construction of $\psi_n'$, \eqref{eq: strong conv in Y_n} and the fact that $\mathsf{PL}_p$ induces the topology of the total variation.\\
For the converse, let us suppose that $(X_n,\mathsf{d}_n,\mu_n)_{n\in \mathbb{N}}$ strongly measured-Gromov converges to the metric measure space $(X_{\infty},\mathsf{d}_{\infty},\mu_{\infty})$. By definition we know that there exist a complete and separable metric space $(X,\mathsf{d})$ and isometric embeddings $\iota_n:X_n\to X$, $n\in \bar{\mathbb{N}}$, such that $(\iota_n)_{\sharp}\mu_n \to (\iota_{\infty})_{\sharp}\mu_{\infty}$ in $\mathscr{M}(X)$ with respect to the topology of the total variation. Since $\mathsf{PL}_p$ metrizes this topology on $\mathscr{M}(X)$ we know that 
$$\mathsf{PL}_p((\iota_n)_{\sharp}\mu_n, (\iota_{\infty})_{\sharp}\mu_{\infty})\rightarrow 0 \qquad \textrm{as} \ n \to \infty$$
and the result follows noticing that $(X,\mathsf{d})$ is a possible competitor in the characterization of $\mathbf{PL}_p$ given in Theorem \ref{th: limit entropy distances}.
\end{proof}

In the next easy proposition we show that the strong measured-Gromov convergence implies  the weak measured-Gromov convergence.

\begin{proposition}
Let $(X_n,\mathsf{d}_n,\mu_n)_{n\in \mathbb{N}}$ be a sequence of metric measure spaces strong measured-Gromov converging to $(X_{\infty},\mathsf{d}_{\infty},\mu_{\infty})$. Then $(X_n,\mathsf{d}_n,\mu_n)_{n\in \mathbb{N}}$ weakly measured-Gromov converges to $(X_{\infty},\mathsf{d}_{\infty},\mu_{\infty})$.
\end{proposition}
\begin{proof}
By definition there exist a complete and separable metric space $(X,\di)$ and isometric embeddings $\iota_n:X_n\to X$, $n\in \bar{\mathbb{N}}$, such that $(\iota_n)_{\sharp}\mu_n \to (\iota_{\infty})_{\sharp}\mu_{\infty}$ in $\mathscr{M}(X)$ with respect to the total variation topology, which implies that $(\iota_n)_{\sharp}\mu_n \to (\iota_{\infty})_{\sharp}\mu_{\infty}$ with respect to the weak convergence. The result follows by the very definition of weak measured-Gromov convergence.
\end{proof}

We conclude the section with a list of examples of convergences.
\begin{examples*}
\begin{enumerate}
\item Let us consider the metric measure space $(X_{\infty},\di_{\infty},\mu_{\infty})$ defined as the unit interval $X_{\infty}=[0,1]$ endowed with the Euclidean distance and the Lebesgue measure. We know that $(X_{\infty},\di_{\infty},\mu_{\infty})$ can be approximated in the weak measured-Gromov convergence by a sequence of discrete spaces: take for instance $X_n=\{m/n\}_{m=0}^{n-1}$ endowed with the distance $\di_n$ inherited from the ambient $1$-dimensional Euclidean space and the measure $\mu_n$ such that $\mu_n(m/n)=1/n$ for every $m=0,...,n-1.$

We next claim that $(X_n,\di_n,\mu_n)$ does not converge to $(X_{\infty},\di_{\infty},\mu_{\infty})$ in the strong measured-Gromov convergence. 
Indeed, for any metric space $(X,\di)$ such that $X_n$ is isometrically embedded in $X$ via $\iota_n$, $n\in \bar{\mathbb{N}}$, we have 
\begin{equation*}
\begin{aligned}
\mathsf{TV}((\iota_n)_{\sharp}\mu_n,(\iota_{\infty})_{\sharp}\mu_{\infty})=\sup_{A\in \mathscr{B}(X)}|(\iota_{\infty})_{\sharp}\mu_{\infty}(A)-(\iota_n)_{\sharp}\mu_n(A)|\\
\geq \mu_{\infty}\left([0,1]\setminus \left\{\bigcup_{m=0}^{n-1}\,m/n\right\}\right)=1 \qquad \textrm{for any} \ n\in \mathbb{N}.
\end{aligned}
\end{equation*}
\item Let us consider the metric measure space $(X_{\infty},\di_{\infty},\mu_{\infty})$ defined as the unit interval $X_{\infty}=[0,1]$ endowed with the Euclidean distance and the measure $\mu_{\infty}=f\dd \mathcal{L}_{[0,1]}$, where $\mathcal{L}_{[0,1]}$ is the Lebesgue measure on $[0,1]$. Let us define the sequence of metric measure spaces $(X_n,\di_n,\mu_n)$ where $X_n=[0,1-1/n]$, $\di_n$ is the Euclidean distance and $\mu_n=f_n\dd \mathcal{L}_{[0,1-1/n]}.$ Let us suppose that $\tilde{f}_n\rightarrow f$ in $L^1([0,1])$, where $$\tilde{f}_n(x)=\begin{cases} f_n(x) &\textrm{if} \quad 0\le x\le 1-1/n\\
0 &\textrm{if} \quad 1-1/n< x\le 1.
\end{cases}
$$
Then, $(X_n,\di_n,\mu_n)\rightarrow (X_{\infty},\di_{\infty},\mu_{\infty})$ in the strong measured-Gromov convergence. To see this, it is enough to notice that for every $n\in \bar{\mathbb{N}}$ the maps $\iota_n:X_n\rightarrow X_{\infty}$ defined as $\iota_n(x)=x$ provides an isometric embedding such that the convergence $(\iota_n)_{\sharp}\mu_n \to (\iota_{\infty})_{\sharp}\mu_{\infty}$ with respect the total variation distance is exactly equivalent to $\tilde{f}_n\rightarrow f$ in $L^1([0,1])$.
\item  Let $(X_n,\di_n,\mu_n)$ be the sequence of collapsing flat tori  $S^1\times \frac{1}{n} S^1\subset {\mathbb R}^4$ endowed with the normalized measures  $\mu_n:= n/(4\pi^2)\,  {\rm dvol_{S^1\times \frac{1}{n} S^1}}.$ It is a standard fact that $(X_n,\di_n,\mu_n)$ converges to $(X_{\infty},\di_{\infty},\mu_{\infty})=(S^1, \di_{S^1}, (2\pi)^{-1} {\mathcal L}^1)$ in the weak measured-Gromov sense (this a standard example of a \emph{collapsing sequence}).
\\ We claim that the convergence cannot be improved to  strong measured-Gromov. 
Indeed, for any metric space $(X,\di)$ such that $X_n$ is isometrically embedded in $X$ via $\iota_n$, $n\in \bar{\mathbb{N}}$, we have 
\begin{equation*}
\begin{aligned}
\mathsf{TV}((\iota_n)_{\sharp}\mu_n,(\iota_{\infty})_{\sharp}\mu_{\infty})=\sup_{A\in \mathscr{B}(X)}|(\iota_{\infty})_{\sharp}\mu_{\infty}(A)-(\iota_n)_{\sharp}\mu_n(A)|\\
\geq \mu_n \left( \Big(S^1\times \frac{1}{n} S^1 \Big) \setminus \gamma_n\big(S^1\big) \right)   =1 \qquad \textrm{for any} \ n\in \mathbb{N},
\end{aligned}
\end{equation*}
where $\gamma_n:S^1\to S^1\times \frac{1}{n} S^1$ is an arbitrary isometric immersion.

\end{enumerate}
\end{examples*}

\subsection{Sturm's distances}
We notice that the classical $p$-Wasserstein distance $\mathcal{W}_p$, $p\ge 1$, can be recovered as a particular case of Entropy-Transport problem with the choices 
\begin{equation}\label{scelte costo-entropia sturm}
F(s)=I_{1}(s):=\begin{cases} 0 \ &\mathrm{if} \ s=1 \\ +\infty &\mathrm{otherwise} \end{cases} 
\qquad \boldsymbol{\mathrm c}(x_1,x_2)=\mathsf{d}^p(x_1,x_2).
\end{equation}

It is clear that $\mathcal{W}_p$ is not a \emph{regular} Entropy-Transport distance, however we show now that we can recover the $\mathbf{D}_p$-distance of Sturm (defined in Definition \ref{def: D}) as a limiting case of our framework.

\begin{theorem}\label{th: limit sturm}
Fix $p\ge 1$ and let us consider the cost function $\ell(d):=d^p$, the entropy function $F:=U_1$, and the power $a:=1/p$. Let us denote by $\mathsf{D}_{\et,n}$ the power $a$ of the Entropy-Transport cost induced by $F_n:=nU_1$ and $\boldsymbol{\mathrm c}=\ell(d)$. Then, for every metric measure spaces $(X_1,\mathsf{d}_1,\mu_1)$, $(X_2,\mathsf{d}_2,\mu_2)\in \boldsymbol{\mathrm{X}}$ 
\begin{equation}\label{eq: limit sturm distances}
\mathcal{D}_p\left((X_1,\mathsf{d}_1,\mu_1),(X_2,\mathsf{d}_2,\mu_2)\right):=\lim_{n \to \infty}\mathbf{D}_{\et,n}\left((X_1,\mathsf{d}_1,\mu_1),(X_2,\mathsf{d}_2,\mu_2)\right) \quad \textrm{is well defined},
\end{equation}
where $\mathbf{D}_{\et,n}$ denotes the function defined as in Definition \ref{def: D-ET} upon replacing $\Det$ by $\mathsf{D}_{\et,n}$.
Moreover, for every metric measure spaces $(X_1,\mathsf{d}_1,\mu_1)$, $(X_2,\mathsf{d}_2,\mu_2)\in \boldsymbol{\mathrm{X}}_{1,p}$ we have
\begin{equation}\label{eq: conv to sturm distances}
\mathcal{D}_p\left((X_1,\mathsf{d}_1,\mu_1),(X_2,\mathsf{d}_2,\mu_2)\right)=\mathbf{D}_p\left((X_1,\mathsf{d}_1,\mu_1),(X_2,\mathsf{d}_2,\mu_2)\right) \quad p\ge 1.
\end{equation}
\end{theorem}
\begin{proof}
We start by proving that the limit \eqref{eq: limit sturm distances} exists on the set $\boldsymbol{\mathrm{X}}$. 
To see this, we notice that $nF(s)\uparrow I_{1}(s)$ for every $s\in [0,\infty)$. In particular, using the explicit formulation of $\mathbf{D}_{\et,n}$ proved in Proposition \ref{pr: formulazione D-LET} (we remark that we have not used the fact that $\mathbf{D}$ is a distance in the proof of the proposition), we can infer that $\mathbf{D}_{\et,n}$ is nondecreasing so that the limit exists. 

It remains to prove that for every $p\ge 1$ we have $\mathcal{D}_p=\mathbf{D}_p$ on the set $\boldsymbol{\mathrm{X}}_{1,p}$. 
Since for every $n\in \mathbb{N}$, for every complete and separable metric space $(X,\mathsf{d})$ and for every $\mu,\gamma\in \Mi(X)$ we have
$$D_{F_n}(\gamma||\mu)\le D_{I_{1}}(\gamma||\mu),$$
it is clear that $\mathcal{D}_p\le\mathbf{D}_p$. 
For the converse inequality, we know that for every $(X_1,\mathsf{d}_1,\mu_1)$, $(X_2,\mathsf{d}_2,\mu_2)\in \boldsymbol{\mathrm{X}}_{1,p}$ we have
\begin{multline}
\mathbf{D}^p_p\left((X_1,\mathsf{d}_1,\mu_1),(X_2,\mathsf{d}_2,\mu_2)\right)\ge \mathcal{D}^p_p\left((X_1,\mathsf{d}_1,\mu_1),(X_2,\mathsf{d}_2,\mu_2)\right)\\
\ge \mathbf{D}^p_{\et,n}((X_1,\mathsf{d}_1,\mu_1),(X_2,\mathsf{d}_2,\mu_2))=\sum_{i=1}^2D_{F_n}(\gamma_{n,i}||\mu_i)+\int_{X_1\times X_2}\hat{\mathsf{d}}^p_n(x,y)\,\dd\boldsymbol{\gamma_n},
\end{multline}
for some $\boldsymbol{\gamma}_n\in \Mi(X_1\times X_2)$ and metric coupling $\hat{\mathsf{d}}_n$ of $\mathsf{d}_1$ and $\mathsf{d}_2$, whose existence is a consequence of Lemma \ref{lem:esistenza ottimo} (notice that we have only used the properties of the cost and the entropy in the proof of the lemma, while the fact that $\mathsf{D}_{\et}$ is a metric plays no role).
Since $D_{F_n}$ is bounded from above by the superlinear entropy $D_{I_{1}}$, by using Lemma \ref{lem: proprieta funzionale entropia} and Lemma \ref{lem: compattezza sottolivelli entropia} we can infer that $\boldsymbol{\gamma}_n$ is weakly converging (up to subsequence) to a limit $\boldsymbol{\gamma}\in \Mi(X_1\times X_2)$ and 
$$\liminf \sum_{i=1}^2D_{F_n}(\gamma_{n,i}||\mu_i) \ge \sum_{i=1}^2 D_{I_{1}}(\gamma_i||\mu_i).$$
By reasoning as in step 2 and step 3 of Lemma \ref{lem:esistenza ottimo}, we know that there exists a pseudo-metric coupling $\hat{\mathsf{d}}$ of $\mathsf{d}_1$ and $\mathsf{d}_2$ such that $\hat{\mathsf{d}}_n$ converges (up to subsequence) pointwise to $\hat{\mathsf{d}}$ and the convergence is uniform on compact sets. Moreover, we have that
\begin{equation}
\liminf_n \int_{X_1\times X_2}\hat{\mathsf{d}}^p_n(x,y)\,\dd\boldsymbol{\gamma}_n \geq \int_{X_1\times X_2}\hat{\mathsf{d}}^p(x,y)\dd\boldsymbol{\gamma},
\end{equation}
and the result follows since $\hat{\mathsf{d}}$ and $\boldsymbol{\gamma}$ are competitors in the explicit formulation of $\mathbf{D}_p$ (see \cite[Lemma 3.3]{St1}) as a consequence of the fact that 
$$D_{I_{1}}(\gamma_i||\mu_i)<\infty \Longleftrightarrow \ \gamma_i=\mu_i, \  i=1,2.$$
\end{proof}

\begin{remark}\label{rem: sturm regular ET}
We point out that we are not claiming that the sequence $\mathsf{D}_{\et,n}$ defined in Theorem \ref{th: limit sturm} is a sequence of \emph{regular} Entropy-Transport distances. Actually, this is the case for $p=2$ as a conseguence of \cite[Theorem 7.25]{LMS}, noticing that the cost of the Entropy-Transport problem induced by $(nU_1,\di^2)$ is $n$ times the cost of the Entropy-Transport problem induced by $(U_1,(\mathsf{d}/\sqrt{n})^2)$ and $\mathsf{d}/\sqrt{n}$ is trivially a complete and separable distance. 
\\In this situation, one can show that $\mathbf{D}_2$ defines a metric possibly attaining the value $+\infty$ on the whole set $\boldsymbol{\mathrm{X}}$, by reasoning as in the proof of Theorem \ref{th: limit entropy distances}.
\end{remark}

\subsection{Piccoli-Rossi distance}

A natural extension of the $\mathcal{W}_1$-metric in the context of Entropy-Transport problem is the Piccoli-Rossi \emph{generalized Wasserstein distance} $\mathsf{BL}$ \cite{PR1,PR2}, induced by the choices
\begin{equation}\label{def choice PR}
F(s)=|s-1|, \qquad \boldsymbol{\mathrm c}(x_1,x_2)=\mathsf{d}(x_1,x_2).
\end{equation}
We notice that the entropy function is not superlinear.

It is proved in \cite{PR1} that $\mathsf{BL}$ is a complete distance on $\Mi(X)$ for every Polish space $(X,\di)$ (\cite{PR1} is in the Euclidean setting, however  the proof  for a Polish space can be performed verbatim).
 
By exploiting the dual formulation of this distance, we know that $\mathsf{BL}$ corresponds to the so-called \emph{flat metric} or \emph{bounded Lipschitz distance} (see \cite[Theorem 2]{PR2}), namely
\begin{equation}
\mathsf{BL}(\mu_1,\mu_2)=\sup\left\{\int_X f\, \dd(\mu_1-\mu_2)\, : \, \| f \|_{\infty}\leq 1, \, \|f\|_{\mathsf{Lip}}\le 1 \right\} \qquad \textrm{for any} \ \mu_1,\mu_2\in \Mi(X).
\end{equation}

We also recall this useful lemma, which is proved in \cite[Proposition 1]{PR1}.
\begin{lemma}\label{lem PR}
Given $\mu_1,\mu_2\in \Mi(X)$, let us consider the Entropy-Transport problem induced by $(F,\boldsymbol{\mathrm c})$ defined in \eqref{def choice PR}. Then the infimum of the problem \eqref{def: problema di trasporto entropico} is attained by a measure $\boldsymbol{\gamma}\in \Mi(X\times X)$ such that $\gamma_i:=(\pi^i)_{\sharp}\boldsymbol{\gamma}\le \mu_i$, $i=1,2$.
\end{lemma}

We have the following: 

\begin{theorem}\label{th: limit PR}
Fix $a=1$, $\ell(d):=d$ and let us consider the sequence $(F_n)_{n\ge 2}$ defined by
$$F_n(s):=\begin{cases} |s-1| \qquad &\textrm{if} \ \ 0\le s \le n \\ \frac{(s-1)^2}{n-1} \qquad &\textrm{if} \ \ s>n. \end{cases}$$
Let us denote by $\mathsf{D}_{\et,n}$ the Entropy-Transport cost induced by $a$, $F_n$ and $\boldsymbol{\mathrm c}=\ell(d)$. 

Then, for every metric measure spaces $(X_1,\mathsf{d}_1,\mu_1)$, $(X_2,\mathsf{d}_2,\mu_2)\in \boldsymbol{\mathrm{X}}$ the quantity
\begin{equation}\label{eq: limit PR distances}
\mathbf{BL}\left((X_1,\mathsf{d}_1,\mu_1),(X_2,\mathsf{d}_2,\mu_2)\right):=\mathbf{D}_{\et,n}\left((X_1,\mathsf{d}_1,\mu_1),(X_2,\mathsf{d}_2,\mu_2)\right) \quad \textrm{is well defined},
\end{equation}
where $\mathbf{D}_{\et,n}$ denotes the function defined as in Definition \ref{def: D-ET} upon replacing $\Det$ by $\mathsf{D}_{\et,n}$.

Moreover, $\mathbf{BL}$ defines a complete metric on $\boldsymbol{\mathrm{X}}$ such that
\begin{equation}\label{eq: conv to PR distances}
\mathbf{BL}\left((X_1,\mathsf{d}_1,\mu_1),(X_2,\mathsf{d}_2,\mu_2)\right)=\inf \mathsf{BL}(\psi^1_{\sharp}\mu_1,\psi^2_{\sharp}\mu_2),
\end{equation}
where the infimum in the right hand side is taken over all complete and separable metric spaces $(\hat{X},\hat{\mathsf{d}})$ with isometric embeddings $\psi^1:\mathsf{supp}(\mu_1)\rightarrow \hat{X}$ and $\psi^2:\mathsf{supp}(\mu_2)\rightarrow \hat{X}$.
\end{theorem}
\begin{proof}
We notice that $(F_n)_{n\ge 2}$ is a sequence of continuous superlinear entropy functions. We also know that $F_n(s)=|s-1|$ in $[0,1]$ and $F_n(s)\geq |s-1|$ in $[0,\infty)$ for every $n\ge 2$, which implies that $\mathsf{D}_{\et,n}$ coincide with $\mathsf{BL}$ thanks to Lemma \ref{lem PR}. In particular we see that $\mathbf{D}_{\et,n}$ does not depend on $n$ and also the identity \eqref{eq: conv to PR distances} follows.

The fact that $\mathbf{BL}$ is a complete distance on $\boldsymbol{\mathrm{X}}$ is a consequence of the completeness of $\mathsf{BL}$ (and thus $\mathsf{D}_{\et,n}$) on the set of measures $\Mi(X)$, and can be proved along the lines of Step 2 in the proof Theorem \ref{th main}. 
\end{proof}

\begin{remark}
We observe that the sequence $\mathsf{D}_{\et,n}$ defined in Theorem \ref{th: limit PR} is not a sequence of \emph{regular} Entropy-Transport distances. The problem here is that the topology induced by the distance $\mathsf{BL}$ does not coincide with the weak topology, but it requires an additional tightness condition (see \cite[Theorem 3]{PR1} for all the details).
\end{remark}

\subsection{Bounds between distances}
The aim of this last short section is to give some explicit bounds between the distances discussed in the paper. 
\begin{proposition}
Let us denote by $\hk$, $\ghk$, $\mathsf{QPL}_p$ (for $1<p\le 3$) and $\mathsf{LPL}_p$ (for $p>1$) the regular Entropy-Transport distances defined in examples (\ref{def:HK}), (\ref{def:GHK}), (\ref{def:power-like}) and (\ref{def:linear power-like}) respectively. Accordingly, we denote by $\mathbf{D}_{\hk}$, $\mathbf{D}_{\ghk}$, $\mathbf{D}_{\mathsf{QPL}_p}$ and $\mathbf{D}_{\mathsf{LPL}_p}$ the induced Sturm-Entropy-Transport distances.
The following inequalities hold:
\begin{enumerate}
\item $\mathbf{D}_{\ghk}\le \mathbf{D}_{\hk}$.
\item $\mathbf{D}_{\mathsf{QPL}_p}\le \mathbf{D}_{\ghk}\le \sqrt{p}\,\mathbf{D}_{\mathsf{QPL}_p} \qquad 1<p\le 3$. 
\item $\mathbf{D}_{\mathsf{LPL}_p}\le \mathbf{PL}_p \qquad p>1.$

\vspace{2mm}
Moreover, for every regular entropy transport distance $\Det$ induced by $(1/p, F, \ell)$ where $p\ge 1$, $F\in \Gamma_0(\mathbb{R}_+)$, $\ell(d)=d^p$ we have:
\item $\DET\le \mathbf{D}_p \qquad p\ge 1.$
\end{enumerate}
\end{proposition}
\begin{proof}
\begin{enumerate}
\item is a consequence of the bound proved in \cite[Section 7.8]{LMS}. 
\item follows by the corresponding inequality proved in \cite[Proposition 7]{DepIeee}.
\item has been shown along the lines of the  proof of Theorem \ref{th: limit entropy distances} (notice that $\mathbf{D}_{\mathsf{LPL}_p}$ equals $\mathbf{D}_{p,1}$ in the notation of that Theorem).
\item is a consequence of the explicit formulations of $\mathbf{D}_p$ and $\DET$, by noticing that for any $F\in \Gamma_0(\mathbb{R}_+)$ we have $F\le I_{1}$ where $I_{1}$ has been defined in \eqref{scelte costo-entropia sturm}.
\end{enumerate}
\end{proof}

\section{Comparison with conic Gromov-Wasserstein}\label{sec: CGW}
Let $\Det$ be a regular entropy transport distance induced by $(a,F,\ell)$. Recalling the construction introduced in Remark \ref{rem: conical formulation}, given $F$ and a number $c\ge 0$ we can associate to the Entropy-Transport problem a function $H_c:[0,+\infty)\times[0,+\infty)\rightarrow [0,+\infty]$ called  marginal perspective function. Moreover, for any complete and separable metric space $(X,\di)$ the function 
$$H:\mathfrak{C}(X)\times \mathfrak{C}(X)\rightarrow [0,+\infty], \qquad H([x,r];[y,s]):=H_{\ell(\di(x,y))}(r,s)\, ,$$
is such that $H^a$ is a distance on $\mathfrak{C}(X)$. In particular, we have
\begin{equation}\label{eq: triang H}
H^a_{\ell(w_3)}(r,t)\le H^a_{\ell(w_1)}(r,s)+H^a_{\ell(w_2)}(s,t)\, , \qquad \textrm{for any} \, r,s,t\in [0,\infty),
\end{equation}
and for any $w_1,w_2,w_3\in [0,\infty)$ such that there exists a complete and separable metric space $(X,\di)$ with $w_1=\di(x_1,x_2)$, $w_2=\di(x_2,x_3)$, $w_3=\di(x_1,x_3)$, $x_1,x_2,x_3\in X$.

We also recall that $H$ is positively $1$-homogeneous in the scalar variables, i.e. 
$$H_c(\lambda r,\lambda s)=\lambda H_c(r,s) \quad \textrm{for any} \ \lambda\ge0,\, c\ge 0,\, r,s\in [0,\infty).$$

Let $X$ be a metric space and fix $\bar{x}\in X$. We define the canonical projection $\mathfrak{p}:X\times [0,\infty)\rightarrow \mathfrak{C}(X)$ as $\mathfrak{p}(x,r)=[x,r].$ We also introduce the maps 
\begin{align}
&\mathsf{r}:\con(X)\rightarrow [0,+\infty), \hspace{2cm} \mathsf{r}[x,r]:=r, \\
&\mathsf{x}:\con(X)\rightarrow X, \hspace{2.9cm} \mathsf{x}[x,r]:=\begin{cases} x \ \textrm{if} \ r>0, \\
\bar{x} \ \textrm{if} \ r=0. \end{cases}
\end{align}

We denote by $\boldsymbol{\mathfrak{y}}=(\mathfrak{y}_1,\mathfrak{y}_2)=([x_1,r_1],[x_2,r_2])$ a point on $\mathfrak{C}(X_1)\times \mathfrak{C}(X_2)$, and we set $\mathsf{r}_i(\boldsymbol{\mathfrak{y}}):=\mathsf{r}(\mathfrak{y}_i)$, $\mathsf{x}_i(\boldsymbol{\mathfrak{y}}):=\mathsf{x}(\mathfrak{y}_i)$.

Given $p\ge 1$, the \emph{$p$-homogeneous marginals} of a measure $\boldsymbol{\alpha}\in\mathscr{M}(\mathfrak{C}(X_1)\times \mathfrak{C}(X_2))$ are defined as 
$$\mathfrak{h}^p_i(\boldsymbol{\alpha}):=(\mathsf{x}_i)_{\sharp}(\mathsf{r}_i^p\boldsymbol{\alpha}), \qquad i=1,2.$$
\\Following the approach of \cite{SVP}, we define the conic Gromov-Wasserstein distance $\mathrm{CGW}^a$ between two metric measure spaces $(X_1,\mathsf{d}_1,\mu_1)$, $(X_2,\mathsf{d}_2,\mu_2)$, as the power $a$ of \begin{equation}\label{def: CGW}
\mathrm{CGW}\big((X_1,\mathsf{d}_1,\mu_1),(X_2,\mathsf{d}_2,\mu_2)\big):=\inf_{\boldsymbol{\alpha}\in \mathcal{U}_p(\mu_1,\mu_2)}\mathcal{H}(\boldsymbol{\alpha})
\end{equation}
where
\begin{equation}
\mathcal{H}(\boldsymbol{\alpha}):=\int\int H_{\ell(|\mathsf{d}_1(x,x')-\mathsf{d}_2(y,y')|)}\big((rr')^p,(ss')^p\big)\,\dd \boldsymbol{\alpha}([x,r],[y,s])\,\dd \boldsymbol{\alpha}([x',r'],[y',s']),
\end{equation}
and 
\begin{equation}
\mathcal{U}_p(\mu_1,\mu_2):=\Big\{\boldsymbol{\alpha}\in \mathscr{M}(\mathfrak{C}(X_1)\times\mathfrak{C}(X_2))\, : \, \mathfrak{h}^p_i(\boldsymbol{\alpha})=\mu_i,\, i=1,2\Big\}.
\end{equation}
In \cite{SVP}, the following main result has been obtained:
\begin{theorem}[{\cite[Theorem 1]{SVP}}]
If $H^a$ is a distance on the cone then the conic Gromov-Wasserstein distance $\mathrm{CGW}^a$ is a metric on $\boldsymbol{\mathrm{X}}.$
\end{theorem}

As we will see, it is possible to prove an inequality between the conic Gromov-Wasserstein distance and the Sturm-Entropy-Transport distance. We start with a lemma.
\begin{lemma}\label{lem: opt conical coupling}
Let $\Det$ be a regular entropy transport distance induced by $(a,F,\ell)$ and let $p\ge 1$. Let $(X_1,\di_1)$, $(X_2,\di_2)$ be two complete and separable metric spaces and let $\hat{\di}$ be a pseudo-metric coupling between $\di_1$ and $\di_2$. For any $\mu_1\in \mathscr{M}(X_1)$ and $\mu_2\in \mathscr{M}(X_2)$ we have
\begin{align}\label{eq: H equiv ET}
\begin{split}
\min_{\boldsymbol{\alpha}\in\mathcal{U}_p(\mu_1,\mu_2)}&\int_{\mathfrak{C}(X_1)\times\mathfrak{C}(X_2)} H_{\ell(\hat{\di}(x,y))}(r^p,s^p)\,\dd \boldsymbol{\alpha}([x,r],[y,s])\\
&=\min_{\boldsymbol{\gamma}\in\mathscr{M}(X_1\times X_2)}\sum_{i=1}^2D_F(\gamma_i||\mu_i)+\int_{X_1\times X_2}\ell\big(\hat{\mathsf{d}}(x,y)\big)\dd\boldsymbol{\gamma}.
\end{split}
\end{align}
In particular, for any pair of metric measure spaces $(X_1,\mathsf{d}_1,\mu_1)$, $(X_2,\mathsf{d}_2,\mu_2)$ there exists $\boldsymbol{\alpha}\in \mathcal{U}_p(\mu_1,\mu_2)$ such that
\begin{equation}\label{eq: opt H-DET}
\int_{\mathfrak{C}(X_1)\times\mathfrak{C}(X_2))} H_{\ell(\hat{\di}(x,y))}(r^p,s^p)\,\dd \boldsymbol{\alpha}([x,r],[y,s])=\DET^{1/a}((X_1,\mathsf{d}_1,\mu_1),(X_2,\mathsf{d}_2,\mu_2))\, ,
\end{equation}
where $\hat{\di}$ is an optimal pseudo-metric coupling for $\DET((X_1,\mathsf{d}_1,\mu_1),(X_2,\mathsf{d}_2,\mu_2))$.
\end{lemma}
\begin{proof}
Notice that under our assumptions we are in the basic \emph{coercive} setting described in \cite[Section 3.1]{LMS}. The equality stated in \eqref{eq: H equiv ET}, corresponding to the equivalence between the homogeneous formulation based on the function $H$ and the primal Entropy-Transport formulation, is thus a consequence of \cite[Theorem 5.8 $(iii)$]{LMS} and the use of the projection map 
$$\boldsymbol{\mathfrak{p}}:X_1\times \mathbb{R}_{+}\times X_2\times \mathbb{R}_{+}\rightarrow \mathfrak{C}(X_1)\times\mathfrak{C}(X_2), \quad \boldsymbol{\mathfrak{p}}(x,r,y,s):=(\mathfrak{p}\otimes\mathfrak{p})\big((x,r),(y,s)\big)=([x,r],[y,s])$$
for passing to the cone (see \cite[Section 7]{LMS} for all the details in the case of the Hellinger-Kantorovich distance, the general case follows straightforwardly).

Once \eqref{eq: H equiv ET} has been proved, \eqref{eq: opt H-DET} can be deduced by recalling the characterization of $\DET$ given in the point $(i)$ of Lemma \ref{lem:esistenza ottimo}.
\end{proof}

We can now state the main result of this section. It should be compared with \cite[Theorem 5.1]{Memoli1}, \cite[Proposition 2.6]{St3}.
\begin{proposition}
Let $\Det$ be a regular entropy transport distance induced by $(a,F,\ell)$ and let $\mathrm{CGW}^a$ be the associated conic Gromov-Wasserstein distance. For any pair of metric measure spaces $(X_1,\mathsf{d}_1,\mu_1)$, $(X_2,\mathsf{d}_2,\mu_2)$, it holds
\begin{equation}
\mathrm{CGW}^a\big((X_1,\mathsf{d}_1,\mu_1),(X_2,\mathsf{d}_2,\mu_2)\big)\le \big(\mu_1(X_1)^a+\mu_2(X_2)^a\big)\,\DET\big((X_1,\mathsf{d}_1,\mu_1),(X_2,\mathsf{d}_2,\mu_2)\big).
\end{equation} 
\end{proposition}
\begin{proof}
By taking advantage of Lemma \ref{lem:esistenza ottimo}, let us consider an optimal pseudo-metric coupling $\hat{\di}$ for $\DET\big((X_1,\mathsf{d}_1,\mu_1),(X_2,\mathsf{d}_2,\mu_2)\big)$. Thanks to Lemma \ref{lem: opt conical coupling}, let us also consider $\boldsymbol{\alpha}\in \mathcal{U}_p(\mu_1,\mu_2)$ satisfying
\begin{equation}
\int H_{\ell(\hat{\di}(x,y))}(r^p,s^p)\,\dd \boldsymbol{\alpha}([x,r],[y,s])=\DET^{1/a}((X_1,\mathsf{d}_1,\mu_1),(X_2,\mathsf{d}_2,\mu_2)).
\end{equation}
To shorten the notation, let us set $\boldsymbol{\mathfrak{y}}:=([x,r],[y,s])$ and $\boldsymbol{\mathfrak{y}}':=([x',r'],[y',s'])$. 
By the triangle inequality for $\hat{\di}$, and the fact that $\ell(\cdot)$ and $H_{\cdot}(r,t)$ are increasing for any $r,t$, we have
\begin{align}\label{dis: CGW-DET}
\nonumber&\mathrm{CGW}^a\big((X_1,\mathsf{d}_1,\mu_1),(X_2,\mathsf{d}_2,\mu_2)\big) \\
\nonumber&\le\bigg(\int\int H_{\ell(\hat{\di}(x,y)+\hat{\di}(x',y'))}\big((rr')^p,(ss')^p\big)\,\dd \boldsymbol{\alpha}(\boldsymbol{\mathfrak{y}})\,\dd \boldsymbol{\alpha}(\boldsymbol{\mathfrak{y}}')\bigg)^a \\
\nonumber&\le\bigg(\int\int \Big[H^a_{\ell(\hat{\di}(x,y))}\big((rr')^p,(r's)^p\big)+H^a_{\ell(\hat{\di}(x',y'))}\big((r's)^p,(ss')^p\big)\Big]^{\frac{1}{a}}\dd \boldsymbol{\alpha}(\boldsymbol{\mathfrak{y}})\,\dd \boldsymbol{\alpha}(\boldsymbol{\mathfrak{y}}')\bigg)^a \\
\nonumber&\le \bigg(\int\int H_{\ell(\hat{\di}(x,y))}(r^p,s^p)(r')^p\,\dd \boldsymbol{\alpha}(\boldsymbol{\mathfrak{y}})\,\dd \boldsymbol{\alpha}(\boldsymbol{\mathfrak{y}}')\bigg)^a+\\
&\hspace{4cm} \bigg(\int\int H_{\ell(\hat{\di}(x',y'))}\big((r')^p,(s')^p\big)s^p\,\dd \boldsymbol{\alpha}(\boldsymbol{\mathfrak{y}})\,\dd\boldsymbol{\alpha}(\boldsymbol{\mathfrak{y}}')\bigg)^a
\end{align}
where we have also used the Minkowski inequality and the homogeneity of $H$ in the last passage, and we have taken advantage of \eqref{eq: triang H} with $w_3=\hat{\di}(x,y)+\hat{\di}(x',y')$, $w_1=\hat{\di}(x,y)$ and $w_2=\hat{\di}(x',y')$. To justify the use of \eqref{eq: triang H} we can argue as follows: if $\hat{\di}(x,y)=0$ or $\hat{\di}(x',y')=0$ there is nothing to prove, otherwise we notice that the three points metric space $(\{A,B,C\},\tilde{\di})$ with mutual distances between different points defined as $\tilde{\di}(A,B):=\hat{\di}(x,y)$, $\tilde{\di}(B,C):=\hat{\di}(x',y')$, $\tilde{\di}(A,C):=\hat{\di}(x,y)+\hat{\di}(x',y')$ is indeed a complete and separable metric space for any $x,x'\in X_1$, $y,y'\in X_2$.
\\ Using the definition of $\boldsymbol{\alpha}\in \mathcal{U}_p(\mu_1,\mu_2)$, we can now perform the integrals in \eqref{dis: CGW-DET} obtaining
\begin{equation*}
\begin{aligned}
& \int\int H_{\ell(\hat{\di}(x,y))}(r^p,s^p)(r')^p\,\dd \boldsymbol{\alpha}([x,r],[y,s])\,\dd \boldsymbol{\alpha}([x',r'],[y',s'])\\
&\qquad  \qquad =\mu_1(X)\int H_{\ell(\hat{\di}(x,y))}(r^p,s^p)\,\dd \boldsymbol{\alpha}([x,r],[y,s])
\\
&\qquad  \qquad = \mu_1(X_1)\,\DET^{1/a}((X_1,\mathsf{d}_1,\mu_1),(X_2,\mathsf{d}_2,\mu_2))\, ,
\end{aligned}
\end{equation*} 
and similarly 
\begin{equation*}
\begin{aligned}
& \int\int H_{\ell(\hat{\di}(x',y'))}\big((r')^p,(s')^p\big)s^p\,\dd \boldsymbol{\alpha}([x,r],[y,s])\,\dd\boldsymbol{\alpha}([x',r'],[y',s'])\\
&\qquad \qquad  = \mu_2(X_2)\,\DET^{1/a}((X_1,\mathsf{d}_1,\mu_1),(X_2,\mathsf{d}_2,\mu_2))\, .
\end{aligned}
\end{equation*} 
We thus reach the desired conclusion.
\end{proof}

\end{document}